\documentclass[reqno]{amsart}
\usepackage{hyperref}

\newcommand{\mX}{X^{s}_{0}(\Omega)}

\begin{document}
\title[\hfilneg
\hfil Arbitrary  many  positive solutions ]
{Arbitrary  many  positive solutions for a nonlinear problem involving the fractional Laplacian}

\author[J. Zhang\,\,,\,\,X. Liu\hfil
\hfilneg]
{Jinguo Zhang,\quad Xiaochun Liu}

\address{Jinguo Zhang\newline
   School of Mathematics \\
  Jiangxi Normal University \\
  330022 Nanchang, China}
\email{jgzhang@jxnu.edu.cn}

\address{Xiaochun Liu \newline
   School of Mathematics and Statistics \\
  Wuhan University \\
  430072 Wuhan, China}
\email{xcliu@whu.edu.cn}

\thanks{Supported by NSFC Grant No.11371282 and
Natural Science Foundation of Jiangxi (No. 20142BAB211002).}

\begin{abstract}
We establish the existence and multiplicity of positive solutions to the  problems
involving the fractional Laplacian:
\begin{equation*}
\left\{\begin{array}{lll}
&(-\Delta)^{s}u=\lambda u^{p}+f(u),\,\,u>0 \quad &\mbox{in}\,\,\Omega,\\
&u=0\quad &\mbox{in}\,\,\mathbb{R}^{N}\setminus\Omega,\\
\end{array}\right.
\end{equation*}
where $\Omega\subset \mathbb{R}^{N}$ $(N\geq 2)$ is a bounded smooth domain, $s\in (0,1)$,
$p>0$, $\lambda\in \mathbb{R}$ and $(-\Delta)^{s}$ stands for the fractional Laplacian.
When $f$ oscillates near the origin or at infinity,
via the variational argument we prove that the problem has arbitrarily
many positive solutions and the number of solutions to problem is strongly influenced by $u^{p}$ and $\lambda$.
Moreover, various properties of the solutions are also described in $L^{\infty}$- and $\mX$-norms.

 \end{abstract}

\maketitle \numberwithin{equation}{section}
\newtheorem{theorem}{Theorem}[section]
\newtheorem{lemma}{Lemma}[section]
\newtheorem{remark}{Remark}[section]
\newtheorem{proposition}{Proposition}[section]
\newtheorem{corollary}{Corollary}[section]
\newtheorem{definition}{Definition}[section]

\section{Introduction}
In this paper, we will consider the existence and  multiplicity of positive
solutions for the following nonlinear fractional equations
\begin{equation}\label{eq1-1}
\left\{\begin{array}{lll}
&(-\Delta)^{s}u=\lambda u^{p}+f(u), &\mbox{in}\,\,\Omega,\\
&u>0 \quad &\mbox{in}\,\,\Omega,\\
&u=0\quad &\mbox{in}\,\,\mathbb{R}^{N}\setminus\Omega,\\
\end{array}\right.
\end{equation}
where $\Omega\subset \mathbb{R}^{N}$ $(N>2s)$ is a bounded domain with smooth
boundary $\partial\Omega$, $0<s<1$ and $(-\Delta)^{s}$ stands for the fractional Laplacian.
$f:\,[0\,,\,\infty)\to \mathbb{R}$ is a continuous function, while $p>0$ and
$\lambda\in \mathbb{R}$ are some parameters.

Here $(-\Delta)^{s}$ is the fractional Laplacian  operator defined,
up to a normalization factor, by the Riesz potential as
$$-(-\Delta)^{s}u(x)=P. V.\,\int\limits_{\mathbb{R}^{N}}\frac{u(x+y)+u(x-y)-2u(x)}{|y|^{N+2s}}\,dy,
\quad x\in \mathbb{R}^{N},$$
where $s\in (0,1)$ is a fixed parameter. Fractional Sobolev space are well-known since the beginning
of the last century, especially in the framework of harmonic analysis.
More recently, after the paper of Caffarelli and Silvestre \cite{ll2007}, a large amount
of papers were written on problems which involve the fractional diffusion $(-\Delta)^{s}$.
This type of diffusion operators arise in several
 areas such as physics, probability and finance, see, for example, \cite{a2009,b1996}.
 Observe that $s=1$ corresponds to the standard local Laplacian.

One can also define the operator $(-\Delta)^{s}$ through the spectral decomposition of $u$,
in terms of the eigenvalues and eigenfunctions of the Laplacian $-\Delta$
 with homogeneous boundary conditions. As explained in \cite{ll2007} and \cite{xt2010},
 when working in a bounded domain $\Omega\subset\mathbb{R}^{N}$ the fractional Laplacian $(-\Delta)^{s}$
 can still be defined as a Dirichlet-to-Neumann map, and this allows us to connect nonlocal problems
 involving $(-\Delta)^{s}$ to suitable degenerate-singular, local problems defined in one more
 space dimension, that is, for any regular function $u$, the fractional Laplacian $(-\Delta)^{s}$
 acting on $u$ is defined by
 \begin{equation}\label{eq1-2}
 (-\Delta)^{s}u(x)=-\frac{1}{\kappa_{s}}\lim\limits_{t\to 0^{+}}t^{1-2s}\frac{\partial w}{\partial t}(x,t),
 \quad \forall x\in \Omega,\,\,t>0,
 \end{equation}
 where $w=E_{s}(u)$ is $s$-harmonic extension of $u$ and $\kappa_{s}=\frac{2^{1-2s}\Gamma(1-s)}{\Gamma(s)}$.
 Using the definition \eqref{eq1-2}, several results of the fractional Laplacian problems were obtained,
 one can see \cite{bes2012,xt2010,cdds2011,cw2014,ckl2013,ff2014,t2011,z2014} and the reference therein.

Since the fractional operator has nonlocal character, working on
bounded domains imposes that an
appropriate variational formulation of the problem is to consider functions on $\mathbb{R}^{N}$
with the condition $u=0$ in $\mathbb{R}^{N}\setminus\Omega$ replacing the boundary condition $u=0$
on $\partial\Omega$.
Set $$\mX=\{u\in H^{s}(\mathbb{R}^{N}):\,u=0\,\,\text{a.e. in}\,\,\mathbb{R}^{N}\setminus\Omega\},$$
where $H^{s}(\mathbb{R}^{N})$ denotes the usual fractional Sobolev space. We say
that $u\in \mX$ is a weak solution of \eqref{eq1-1} if for every $\varphi\in\mX$, one has
$$\int\limits_{\mathbb{R}^{N}\times\mathbb{R}^{N} }\frac{(u(x)-u(y))(\varphi(x)-\varphi(y))}{|x-y|^{N+2s}}dxdy=
\lambda\int\limits_{\Omega}u^{p}\varphi dx+\int\limits_{\Omega}f(u)\varphi dx.$$

Before starting our main results, we notice that the nonlinear equations involving the fractional
Laplacian have been widely studied recently.
For instance, Chang and Gonz\'{a}lez \cite{cg2011} studied this operator in conformal geometry.
Caffarelli et al.\cite{ljy2010,lsl2008} investigated free boundary problems of the
fractional Laplacian. Since the work of Caffarelli and Silvestre \cite{ll2007},
who introduced the $s$-harmonic extension to define the fractional Laplacian operator,
several results of the fractional Laplacian problems were obtained.
 Chang and Wang \cite{cw2014} obtained some nodal solutions of the fractional Laplacian.
 Barrios et. al. \cite{bes2012} studied the critical fractional problem.
Elliptic equations with fractional Laplacian were also studied by many authors,
see \cite{bes2013,b2014,xt2010,cw2013,sp2013,fw2012,s2006,sr2014-1,sr2014-2,s2013-1,
s2014-1,s2013-2,s2012,t2011,ws2014,zl2014}
and references therein.

The purpose of this paper is to study the number and behavior of solutions to the fractional elliptic
 problem \eqref{eq1-1} while $f$ oscillates near the origin or at infinity.
Using the variational methods we show that the nontrivial
weak solutions to problem \eqref{eq1-1} are strongly influenced by $u^{p}$, $\lambda$
and the oscillatory nonlinearity $f$.
In the past decades, the existence and multiplicity of classical elliptic boundary
value problems have been widely investigated,
on can see \cite{gz2013,kmt2007,km2010,oo2006,oz1996,r2002}.
Specially, the oscillatory terms usually give infinitely many distinct solutions for the equations.
However, due to the fact that the fractional Laplacian operator is nonlocal,
very few thing on this topic is known.
In this paper, we generalize the truncation  methods to the situation of fractional Laplacian
and get the arbitrarily many positive solutions to the elliptic equations involving
fractional Laplacian.

In the sequel, we state our main results by treating separately two cases, i.e.,
when $f$ oscillates near the origin, or at infinity, respectively.
We assume that $f\in C([0,\infty)\,,\,\mathbb{R})$, $f(0)=0$ and satisfies
\begin{itemize}
\item[($f_{0}$)] $-\infty<\liminf\limits_{t\to 0^{+}}\frac{F(t)}{t^2}\leq
              \limsup\limits_{t\to 0^{+}}\frac{F(t)}{t^2}=+\infty$,
          and $l_{0}:=\liminf\limits_{t\to 0^{+}}\frac{f(t)}{t}<0$;

or

\item[($f_{\infty}$)]$-\infty<\liminf\limits_{t\to \infty}\frac{F(t)}{t^2}\leq
              \limsup\limits_{t\to\infty}\frac{F(t)}{t^2}=+\infty$,
and $l_{\infty}:=\liminf\limits_{t\to \infty}\frac{f(t)}{t}<0$,
\end{itemize}
where $F(t)=\int_{0}^{t}f(\tau)d\tau$, $t>0$.

\begin{remark}\label{r1-1}
\begin{itemize}
\item[(1)]Hypotheses $(f_{0})$ implies an oscillatory behavior of $f$ near the origin, and
$(f_{\infty})$ implies an oscillatory behavior of $f$ at infinity. Moreover, there is no increasing condition on $f$,
which means $f$ could be supercritical.
\item[(2)]It follows from the assumption $\liminf\limits_{t\to 0^{+}}\frac{f(t)}{t}<0$ that
there exists a subsequence $\{t_{k}\}\subset(0\,,\,\infty)$ converging to 0 such that $f(t_{k})<0$
for every $k\in \mathbb{N}$. Without loss of generality, we could assume that this sequence is strictly decreasing.
Similarly, from  $\liminf\limits_{t\to \infty}\frac{f(t)}{t}<0$, we will assume there is a strictly
increasing sequence $\{t_{k}\}\subset (0\,,\,\infty)$ converging to $+\infty$ such that $f(t_{k})<0$ for all $k\in \mathbb{N}$.
\end{itemize}
\end{remark}

\begin{remark}\label{r1-3}
\begin{itemize}
\item[(1)] An example of a function $f:\,[0\,,\,\infty)\to \mathbb{R}$ satisfies $(f_0)$ is
defined by
\begin{equation*}
f(t)=\left\{\begin{array}{lll}
&0\quad &\text{if}\quad t=0;\\
&t^{\alpha}(a+\sin\frac{1}{t^{\beta}})\quad &\text{if}\quad t>0,\\
\end{array}\right.
\end{equation*}
where $\alpha$, $\beta$, $a\in \mathbb{R}$ satisfy $0<\alpha<1<\alpha+\beta$ and $0<a<1$.
\item[(2)]
The other function $f:\,[0\,,\,\infty)\to \mathbb{R}$ satisfies $(f_{\infty})$ is
$$f(t)=t^{\alpha}(a+\sin t^{\beta}),$$
where $\alpha$, $\beta$, $a\in \mathbb{R}$ satisfy $\alpha>1$, $|\alpha-\beta|<1$ and $0<a<1$.
\end{itemize}
\end{remark}

In order to formulate our results, we recall some notations. More details will be see in Section 2.
The Hilbert space $X_{0}^{s}(\Omega)$ is endowed with its inner product and norm,
$$\langle u\,,\,v \rangle_{\mX}=\int\limits_{\mathbb{R}^{N}\times\mathbb{R}^{N}}\frac{(u(x)-u(y))(v(x)-v(y))}{|x-y|^{N+2s}}dxdy,
\quad u\,,\,v\in \mX,$$
and $$\|u\|_{\mX}^{2}=\int\limits_{\mathbb{R}^{N}\times\mathbb{R}^{N}}\frac{|u(x)-u(y)|^{2}}{|x-y|^{N+2s}}dxdy,
\quad u\in \mX.$$
The space $L^{q}(\Omega)$ is endowed with its usual norm $\|\cdot\|_{L^{q}(\Omega)}$ for $q\in [1\,,\,+\infty]$.

The first result deals with the case that $f$ is oscillatory near the origin.
\begin{theorem}\label{th1-1}
Assume that $(f_0)$ holds. If
\begin{itemize}
\item[{\em (a)}]either $p=1$ and $\lambda<\lambda_{0}$ for some $0<\lambda_{0}<-l_{0}$,
\item[{\em (b)}] or $p>1$ and $\lambda\in \mathbb{R}$ is arbitrary,
\end{itemize}
then there exist infinitely many positive solutions
 $\{u_{k}\}_{k\in \mathbb{N}}\subset \mX$  of Eq.
\eqref{eq1-1} such that
\begin{equation}\label{eq1-5}
\lim\limits_{k\to +\infty}\|u_{k}\|_{\mX}
=\lim\limits_{k\to +\infty}\|u_{k}\|_{L^{\infty}(\Omega)}=0.
\end{equation}
\end{theorem}

\begin{remark}\label{r1-2}
Notice that $p>1$ may be critical and supercritical in Theorem \ref{th1-1} (b). Having a suitable nonlinearity
oscillating near the origin, Theorem \ref{th1-1} roughly says that the term $u^{p}$
does not affect the number of the solutions to problem \eqref{eq1-1} whenever $p>1$.
\end{remark}

On the other hand, when $0<p<1$, the term $u^p$ may compete with the function $f$ near the origin such
that the number of the solutions to \eqref{eq1-1} becomes finite for many values of $\lambda$.
In this case, we have the following result:

\begin{theorem}\label{th1-2}
Assume $f$ satisfies $(f_0)$ and $0<p<1$. Then there exist $\lambda_{k}>0$
such that problem \eqref{eq1-1} has at least $k$ distinct weak solutions
$u_{1,\lambda}, u_{2,\lambda},\cdot\cdot\cdot,u_{k,\lambda}$ for every
$\lambda\in[-\lambda_{k},\lambda_{k}]$ and $k\in \mathbb{N}$. Moreover,
for any $i\in \{1,2,\cdot\cdot\cdot,k\}$,
\begin{equation}\label{eq1-6}
\|u_{i,\lambda}\|_{\mX}<\frac{1}{i}\quad\text{and}\quad
\|u_{i,\lambda}\|_{L^{\infty}(\Omega)}<\frac{1}{i} .
\end{equation}
\end{theorem}

In the sequel, we will state the counterparts of Theorems \ref{th1-1} and \ref{th1-2}
whenever $f$ oscillates at infinity.
\begin{theorem}\label{th1-3}
Assume that $(f_{\infty})$ holds. If
\begin{itemize}
\item[{\em (a)}]either $p=1$ and $\lambda<\lambda_{\infty}$ for some $0<\lambda_{\infty}<-l_{\infty}$,
\item[{\em (b)}] or $p<1$ and $\lambda\in \mathbb{R}$ is arbitrary,
\end{itemize}
then there exist infinitely many positive solutions  $\{u_{k}\}_{k\in \mathbb{N}}\subset \mX$  of problem
\eqref{eq1-1} such that
\begin{equation}\label{eq1-5}
\lim\limits_{k\to +\infty}\|u_{k}\|_{L^{\infty}(\Omega)}=\infty.
\end{equation}
\end{theorem}

\begin{theorem}\label{th1-4}
Let $f$ satisfies $(f_{\infty})$ and $p>1$. Then there exists $\lambda_{k}>0$
such that problem \eqref{eq1-1} has at least $k$ distinct weak solutions
$u_{1,\lambda}, u_{2,\lambda},\cdot\cdot\cdot,u_{k,\lambda}$ for every
$\lambda\in[-\lambda_{k},\lambda_{k}]$ and $k\in \mathbb{N}$. Moreover,
for any $i\in \{1,2,\cdot\cdot\cdot,k\}$,
\begin{equation}\label{eq1-6}
\|u_{i,\lambda}\|_{L^{\infty}(\Omega)}>i-1.
\end{equation}
\end{theorem}

As we already pointed out, the method developed in the present paper is applicable
in more general setting; not only the type of the domain $\Omega$ can vary with various boundary condition,
but also equations involving the Laplacian and $p$-Laplacian can be considered.
Using the ideal of \cite{ak2008,km2010}, we prove that the nonlocal problem \eqref{eq1-1}
has arbitrarily  many positive solutions.
In addition, we extend the truncation methods of \cite{km2010}
such that it is suitable for the nonlocal elliptic equations.

Since we are looking for positive solutions, we take as usual $f(u)$ defined on all $u\in \mathbb{R}$,
making $f(u)=0$ if $u\leq 0$.
Thus, as we know that the critical points of the functional associated with Eq.\eqref{eq1-1}, i.e.,
 $$\mathcal{J}(u)=\frac{1}{2}\int\limits_{\mathbb{R}^{N}\times\mathbb{R}^{N}}\frac{|u(x)-u(y)|^{2}}{|s-y|^{N+2s}}\,dxdy
 -\frac{\lambda}{p+1}\int\limits_{\Omega} u^{p+1}dx-\int\limits_{\Omega}F(u)dx,$$
are weak solutions of the equation \eqref{eq1-1}, if $u$ is a critical point of $\mathcal{J}$ then
\begin{equation*}\label{eq1-7}
\aligned
0
&=\langle\mathcal{J}'(u)\,,\,u^{-}\rangle\\
&=\int\limits_{\mathbb{R}^{N}\times\mathbb{R}^{N}}\frac{(u(x)-u(y))(u^{-}(x)-u^{-}(y))}{|x-y|^{N+2s}}dxdy
-\lambda\int\limits_{\Omega}u^{p}\,u^{-}dx-\int\limits_{\Omega}f(u)u^{-}dx\\
&\geq \int\limits_{\mathbb{R}^{N}\times\mathbb{R}^{N}}\frac{|u^{-}(x)-u^{-}(y)|^{2}}{|x-y|^{N+2s}}dxdy
-\lambda\int\limits_{\Omega}u^{p}\,u^{-}dx-\int\limits_{\Omega}f(u)u^{-}dx,\\
\endaligned
\end{equation*}
where $u^{-}=\min\{u\,,\,0\}$. This implies that
$$\|u^{-}\|^{2}_{\mX}=\int\limits_{\mathbb{R}^{N}\times\mathbb{R}^{N}}\frac{|u^{-}(x)-u^{-}(y)|^{2}}{|x-y|^{N+2s}}dxdy=0.$$
Thus, necessarily we have $u\geq 0$.
Here we use the following inequality
\begin{equation}\label{eq1-8}
(u(x)-u(y))(u^{-}(x)-u^{-}(y))\geq |u^{-}(x)-u^{-}(y)|^{2}\quad \forall \,\,\,x,y\in\mathbb{R}^{N},
\end{equation}
where $u^{-}(x)=\min\{u(x)\,,\,0\}$. To check \eqref{eq1-8}, since the role of $x$
and $y$ is symmetric, we can always suppose $u(x)\geq u(y)$.
Also \eqref{eq1-8} is clearly an identity when $x,\,y\in \{t:\,u(t)>0\}$ and
when $x,\,y\in\{t:\,u(t)\leq 0\}$. So it only remains to check \eqref{eq1-8} when
$x\in\{t:\,u(t)>0\}$ and $y\in \{t:\,u(t)\leq 0\}$. In this case
$$u^{-}(x)-u^{-}(y)=0-u^{-}(y)=-u(y)\leq  u(x)-u(y).$$
We multiply by $u^{-}(x)-u^{-}(y)=-u(y)\geq 0$ from both side,
and then obtain \eqref{eq1-8}.

This paper is organized as follows.
In Section 2, we introduce a variational setting of the problem
and present some preliminary results.
Section 3 is devoted to study an auxiliary problem.
This section plays an important role in the proof of our main results.
In Section 4, we give the proof of Theorems \ref{th1-1} and \ref{th1-2}.
Finally, the proof of Theorems \ref{th1-3} and \ref{th1-4} are given in Section 5,
and some more general cases are given in Section 6.

\vspace{2mm}

Here we list some notations which will be used throughout the paper.
\begin{itemize}
\item[$\cdot$] The letter $x$ represents a variable in the $\mathbb{R}^{N}$ or in $\Omega$.
\item[$\cdot$]$L^{p}(\Omega)$ $(1\leq p\leq \infty)$ denotes the usual
Sobolev space with norm $\|\cdot\|_{L^{p}(\Omega)}$.
\item[$\cdot$]$|\Omega|$ denotes the Lebesgue measure of the set $\Omega$.
\item[$\cdot$] Big $O$ and small $o$  describe the limit behavior of a certain quantity as $n\to\infty$.
\item[$\cdot$] We denote by $B_{r}(x_{0})=\{x\in \mathbb{R}^{N}:\,|x-x_{0}|<r\}$ the ball at
 each $x_{0}\in \mathbb{R}^{N}$ with radius $r>0$.
\end{itemize}

\section{Preliminaries}
In this section we first recall the background of the fractional Laplacian.
We refer to \cite{berf2014,sr2014-1,sr2014-2,s2013-1,s2013-2,s2012} for the details.

For $s\in (0\,,\,1)$, we denote the classical fractional Sobolev space $H^{s}(\mathbb{R}^{N})$
with the Gagliardo norm
$$\|u\|_{H^{s}(\mathbb{R}^{N}}=\|u\|_{L^{2}(\mathbb{R}^{N})}+
\Big(\int\limits_{\mathbb{R}^{N}\times\mathbb{R}^{N}}\frac{|u(x)-u(y)|^{2}}{|x-y|^{N+2s}}dxdy\Big)^{\frac{1}{2}}.$$

Let $\Omega$ be an open set in $\mathbb{R}^{N}$.
Due to the nonlocal character of the fractional Laplacian, we will consider
the space $\mX$ defined as follows
$$X^{s}_{0}(\Omega)=\{u\in X:\,\,u=0\,\,\text{a.e. in}\,\,\mathbb{R}^{N}\setminus\Omega\}.$$
We refer to \cite{s2012,sv2014,sv2013} for a general definition of $\mX$
and its properties and to \cite{sp2013,v2011} for an account of the properties of $H^{s}(\mathbb{R}^{N})$.

In $\mX$ we can consider the following norm
 $$\|u\|_{\mX}=
\Big(\int\limits_{\mathbb{R}^{N}\times\mathbb{R}^{N}}\frac{|u(x)-u(y)|^{2}}{|x-y|^{N+2s}}dxdy\Big)^{\frac{1}{2}}.$$
 We also recall that $(\mX\,,\,\|\cdot\|_{\mX})$ is a Hilbert space, with scalar product
 $$\langle u\,,\,v\rangle_{\mX}=
\int\limits_{\mathbb{R}^{N}\times\mathbb{R}^{N}}\frac{(u(x)-u(y))(v(x)-v(y))}{|x-y|^{N+2s}}dxdy,\quad \forall u,\,v\in \mX.$$

Observe that by \cite{sp2013} we have the following identity
$$\|u\|_{\mX}=\|(-\Delta)^{\frac{s}{2}}u\|_{L^{2}(\mathbb{R}^{N})},$$
this leads us to establish as a definition for the solution to our problem.
\begin{definition}\label{def1-1}
We say that $u\in \mX$ is a weak solution of \eqref{eq1-1} if
$$\int\limits_{\mathbb{R}^{N}\times\mathbb{R}^{N}}\frac{(u(x)-u(y))(\varphi(x)-\varphi(y))}{|x-y|^{N+2s}}dxdy=
\lambda\int\limits_{\Omega}u^{p}\varphi dx+\int\limits_{\Omega}f(u)\varphi dx$$
holds for every $\varphi\in \mX$.
\end{definition}

 Define the energy function $\mathcal{J}:\,\mX\to \mathbb{R}$ by
 $$\mathcal{J}(u)=\frac{1}{2}\int\limits_{\mathbb{R}^{N}\times\mathbb{R}^{N}}\frac{|u(x)-u(y)|^{2}}{|x-y|^{N+2s}}\,dxdy
 -\frac{\lambda}{p+1}\int\limits_{\Omega} u^{p+1}dx-\int\limits_{\Omega}F(u)dx,$$
where $F(u)=\int_{0}^{u}f(t)dt.$

In order to carry out the nonlinear analysis, from \cite{s2013-2,s2012},
we have the following embedding theorem.
\begin{lemma}\label{l2-2}
\begin{itemize}
\item[{\em (i)}] $C^{2}_{0}(\Omega)\hookrightarrow X_{0}^{s}(\Omega)$
and $\mX\hookrightarrow H^{s}(\mathbb{R}^{N})$;
\item[{\em (ii)}] If $\Omega$ has a Lipschitz boundary, the embedding $\mX\hookrightarrow L^{q}(\Omega)$
is compact for any $q\in [1\,,\,2^{*}_{s})$, and the embedding $\mX\hookrightarrow L^{2^*_{s}}(\Omega)$ is continuous,
 where $2^*_{s}:=\frac{2N}{N-2s}$ ($N>2s$) is critical Sobolev exponent.
\end{itemize}
\end{lemma}

In view of $(f_{0}) $ or $(f_{\infty})$ and the Sobolev embedding $\mX\hookrightarrow L^{\frac{2N}{N-2s}}(\Omega)$,
it is not difficult to see that $\mathcal{J}\in C^{1}$. Moreover, if $u\in \mX$
is a critical points of $\mathcal{J}$, then $u$
is a weak solution of problem \eqref{eq1-1}. The converse is also true.

\section{Some auxiliary results}
In this section, we consider the following generic problem
\begin{equation}\label{eq2-1}
\left\{\begin{array}{lll}
&(-\Delta)^{s}u+\mu u=g(u),\,\,u\geq 0\quad &\text{in}\,\,\Omega,\\
&u=0\quad &\text{in}\,\,\mathbb{R}^{N}\setminus\Omega,\\
\end{array}\right.
\end{equation}
where $\mu>0$ and the function $g$ satisfies
\begin{itemize}
\item[$(g_{1})$] $g:\,\mathbb{R}^{+}\to \mathbb{R}$ is a continuous, bounded function and $g(0)=0$.
\item[$(g_{2})$] There exist $\eta>\delta>0$ such that $g(t)\leq0$ for  all
$t\in [\delta\,,\,\eta]$.
\end{itemize}

Due to $(g_1)$ and looking for the positive solutions, we may extended $g$ continuously to
the whole $\mathbb{R}$ and define $g(u)=0$ for all $u\leq 0$.

The corresponding energy function $\mathcal{E}:\,\mX\to \mathbb{R}$ of the problem \eqref{eq2-1}
is defined as follows:
\begin{equation*}\label{eq2-2}
\mathcal{E}(u)=\frac{1}{2}\int\limits_{\mathbb{R}^{N}\times\mathbb{R}^{N}}\frac{|u(x)-u(y)|^{2}}{|x-y|^{N+2s}}\,dxdy+
\frac{\mu}{2}\int\limits_{\Omega}|u|^{2}dx-\int\limits_{\Omega}G(u)dx,\quad \forall u\in \mX ,
\end{equation*}
 where $G(u)=\displaystyle{\int}^{u}_{0}g(t)dt$.  It is easy to see that $\mathcal{E}$
 is well defined. Indeed, by the mean value theorem and Sobolev embedding theorem,
 for all $u\in \mX$ we have
 \begin{equation}\label{eq2-2*}
 \int\limits_{\Omega}G(u)dx
 =\int\limits_{\Omega}\int\limits_{0}^{u}g(t)dt dx\leq C_{1}\,\sup\limits_{u\in \mX}|g||\Omega|^{\frac{1}{2}}\,\|u\|_{\mX}
 <\infty,
\end{equation}
where $C_{1}>0$  is the constant of Sobolev embedding $\mX\hookrightarrow L^{2}(\Omega)$.
 Moreover, standard arguments show that $\mathcal{E}$ is of class $C^{1}$ on  $\mX$.

 From $(g_1)$ and $(g_2)$, we know that the nonlinear term $g$
grows subcritically.
 Then we have the following result.
 \begin{lemma}\label{l2-1}
If $(g_1)$ and $(g_2)$ hold, then the functional $\mathcal{E}$ is coerciveness,  bounded from  below
and satisfies the {\em(PS)}-condition.
 \end{lemma}
 \begin{proof}
 From \eqref{eq2-2*}, for all $u\in \mX$ we have
\begin{equation*}\label{eq2-3}
\aligned
\mathcal{E}(u)
&=\frac{1}{2}\int\limits_{\mathbb{R}^{N}\times\mathbb{R}^{N}}\frac{|u(x)-u(y)|^{2}}{|x-y|^{N+2s}}\,dxdy
+\frac{\mu}{2}\int\limits_{\Omega}|u|^{2}dx
-\int\limits_{\Omega}G(u)dx\\
&\geq \frac{1}{2}\int\limits_{\mathbb{R}^{N}\times\mathbb{R}^{N}}\frac{|u(x)-u(y)|^{2}}{|x-y|^{N+2s}}\,dxdy
-C_{1}\,\sup\limits_{u\in \mX}|g||\Omega|^{\frac{1}{2}}\,\|u\|_{\mX}\\
&= \frac{1}{2}\|u\|_{\mX}^{2}
-C_{1}\,\sup\limits_{u\in \mX}|g||\Omega|^{\frac{1}{2}}\,\|u\|_{\mX}\\
&\geq \frac{1}{2}\|u\|_{\mX}^{2}-\frac{1}{2}\Big(C_{1}\sup\limits_{u\in \mX}|g||\Omega|^{\frac{1}{2}}\Big)^{2}
-\frac{1}{2}\|u\|_{\mX}^{2}\\
&=-\frac{1}{2}\Big(C_{1}\sup\limits_{u\in \mX}|g||\Omega|^{\frac{1}{2}}\Big)^{2},
\endaligned
\end{equation*}
which implies that the functional $\mathcal{E}$  is coercive and
 bounded from below on  $\mX$.

 Next, we prove the functional $\mathcal{E}$ satisfies the (PS)-condition. Let $\{u_{n}\}_{n\in \mathbb{N}}\subset \mX$
be a (PS)-sequence for $\mathcal{E}$ verifying
 $$\mathcal{E}(u_{n})\to c\,\, \text{and}\,\,\,
  \mathcal{E}'(u_{n})\to 0\,\,\text{as}\,\,n\to \infty.$$
Then, since $\mathcal{E}$ is coercive, there exists $M>0$ such that $\|u_{n}\|_{\mX}\leq M$.
 So, up to a subsequence, we may assume that there exists $u\in \mX$ such that
 \begin{equation}\label{eq2-3-1}
 \int\limits_{\mathbb{R}^{N}\times\mathbb{R}^{N}}\frac{(u_{n}(x)-u_{n}(y))(\varphi(x)-\varphi(y))}{|x-y|^{N+2s}}dxdy
 \to  \int\limits_{\mathbb{R}^{N}\times\mathbb{R}^{N}}\frac{(u(x)-u(y))(\varphi(x)-\varphi(y))}{|x-y|^{N+2s}}dxdy
 \end{equation}
 as $n\to\infty$ for any $\varphi\in \mX$.
 Moreover, by Lemma \ref{l2-2}, up to a subsequence, as $n\to\infty $ we have
 \begin{equation}\label{eq2-29}
 \begin{array}{lll}
 &u_{n}\to u\,\,\, \text{strongly in}\,\,\, L^{q}(\Omega),\,2\leq q<2^*_{s};\\
 &u_{n}(x)\to u(x)\,\,\, \text{a.e. in}\,\,\, \Omega.\\
 \end{array}
 \end{equation}
So by \eqref{eq2-29}, the continuity of $g$  and the Dominated Convergence
Theorem we obtain
\begin{equation}\label{eq2-3-2}
\int\limits_{\Omega}g(u_{n})\,u_{n}\,dx\to
\int\limits_{\Omega}g(u)\,u\,dx\quad \text{as}\,\,n\to \infty,
\end{equation}
and
 \begin{equation}\label{eq2-3-3}
\int\limits_{\Omega}g(u_{n})\,\hat{u}\,dx\to
\int\limits_{\Omega}g(u)\,u\,dx\quad \text{as}\,\,n\to \infty.
\end{equation}
Moreover, by $\mathcal{E}'(u_{n})\to 0$, \eqref{eq2-29} and \eqref{eq2-3-2},
we have
\begin{equation}\label{eq2-3-4}
 \int\limits_{\mathbb{R}^{N}\times\mathbb{R}^{N}}\frac{|u_{n}(x)-u_{n}(y)|^{2}}{|x-y|^{N+2s}}dxdy+\mu \int\limits_{\Omega}|u|^{2}\,dx-
\int\limits_{\Omega}g(u)\,u\,dx=o(1).
\end{equation}
On the other hand, by $\langle\mathcal{E}'(u_{n})\,,\,\hat{u}\rangle=o(1)$,
\eqref{eq2-29} and \eqref{eq2-3-3}, we deduce that
\begin{equation}\label{eq2-3-5}
 \int\limits_{\mathbb{R}^{N}\times\mathbb{R}^{N}}\frac{(u_{n}(x)-u_{n}(y))(u(x)-u(y))}{|x-y|^{N+2s}}dxdy+\mu \int\limits_{\Omega}|u|^{2}\,dx-
\int\limits_{\Omega}g(u)\,u\,dx=o(1).
\end{equation}
Thus, \eqref{eq2-3-1}, \eqref{eq2-3-4} and \eqref{eq2-3-5} give that
$$\int\limits_{\mathbb{R}^{N}\times\mathbb{R}^{N}}\frac{|u_{n}(x)-u_{n}(y)|^{2}}{|x-y|^{N+2s}}dxdy\to
\int\limits_{\mathbb{R}^{N}\times\mathbb{R}^{N}}\frac{|u(x)-u(y)|^{2}}{|x-y|^{N+2s}}dxdy,$$
that is,
\begin{equation}\label{eq2-3-6}
\|u_{n}\|_{\mX}\to \|u\|_{\mX}\quad \text{ as} \quad n\to \infty.
\end{equation}

Combining \eqref{eq2-3-1},\eqref{eq2-3-6}
 and letting $n\to\infty$, we conclude $u_{n}\to u$ strongly in $\mX$, which implies that the functional
 $\mathcal{E}$ satisfies the (PS)-condition
  and $u$ is a solutions of Eq. \eqref{eq2-1}.
 \end{proof}

 From  Lemma \ref{l2-1}, we can consider the minimum problem
$$c:=\inf\limits_{u\in\mX}\mathcal{E}(u).$$
We will find a global minimizer $u$ of functional $\mathcal{E}$ which is a solution of problem \eqref{eq2-1}.

Now we  introduce the set
$$E_{\eta}=\{u\in \mX:\,\,\|u\|_{L^{\infty}(\Omega)}\leq \eta\},$$
where the number $\eta>0$ from $(g_2)$. Then we have the following result.

\begin{theorem}\label{th2-1}
Assume that $(g_1)$ and $(g_2)$ hold. Then
\begin{itemize}
\item[{\em (i)}] the functional $\mathcal{E}$ infimum is attained at some $u_{0}\in E_{\eta}$;
\item[{\em (ii)}] $u_{0}$ is a weak solution of problem \eqref{eq2-1} and $u_{0}\in [0\,,\,\delta]$,
where $\delta$ is given in $(g_2)$.
\end{itemize}
\end{theorem}

\begin{proof}
(i) Now, we claim that $\mathcal{E}$ is sequentially
weak lower semi-continuous. The claim follows at once if we prove that $u\mapsto \int_{\Omega}G(u)dx$,
$\forall \,u\in \mX$, is sequentially weakly continuous. Arguing by contradiction, we assume that
there exist $\varepsilon_{0}>0$ and a sequence
$\{u_{n}\}_{n\in \mathbb{N}}\subset\mX$  which converges weakly to $u\in \mX$,
up to a subsequence, such that
\begin{equation*}\label{eq2-2**}
\Big|\int\limits_{\Omega}\Big(G(u_{n})-G(u)\Big)dx\Big|\geq\varepsilon_{0}>0 \quad \text{for all}\,\,n\in \mathbb{N}.
\end{equation*}
On the other hand, using the mean value theorem,
H\"{o}lder inequality and Sobolev embedding theorem,
we deduce that
$$ |\int\limits_{\Omega}\Big(G(u_{n})-G(u)\Big)dx|
\leq \sup\limits_{u\in \mX}|g|\cdot|\Omega|^{\frac{1}{2}}\Big(\int\limits_{\Omega}|u_{n}-u|^{2}\Big)^{\frac{1}{2}}\to 0$$
as $n\to \infty$, which contradicts $\varepsilon_{0}>0$.

Note that the set $E_{\eta}$ is convex and closed in $\mX$,
and thus weakly closed in $\mX$.
Then there exists $u_{0}\in E_{\eta}$ which is a minimum point
of $\mathcal{E}$ over $E_{\eta}$ since $\mathcal{E}$ is coercive on $E_{\eta}$ . This proves (i).

(ii)
Firstly, we prove $u_{0}\in [\,0\,,\,\delta\,]$.
Arguing by  contradiction, we assume the set
 $$\widetilde{\Omega}:=\{x\in \Omega:\,\,u_{0}<0\,\,\,\text{or}\,\,\,u_{0}>\delta\}$$ is not empty,
and suppose that $|\widetilde{\Omega}|>0$, where $|\widetilde{\Omega}|$ denotes the Lebesgue measure.

Define $\gamma:\,\mathbb{R}\to \mathbb{R}$ by
$$\gamma(t)=\min\{\,t^{+}\,,\,\delta\,\},$$
where $t^{+}=\max\{\,t\,,\,0\,\}$. Then
$\gamma$ is a Lipschitz function, and $\gamma(t)\geq 0$ and $\gamma(0)=0$.

Now, set $u=\gamma\circ u_{0}$, we have $u\in \mX$.
Moreover, for a.e. $x\in \widetilde{\Omega}$,
 $$u(x)=\gamma(u_{0}(x))=\min\{u_{0}^{+}(x)\,,\,\delta\}\leq \delta,$$
and $0\leq u(x)\leq \delta<\eta$.
So $u\in E_{\eta}$.

Define the sets
$$\Omega_{1}:=\{x\in \Omega:\,\,u_{0}(x)<0\}\quad \text{and}
\quad\Omega_{2}:=\{x\in \Omega:\,\,u_{0}(x)>\delta\}.$$
Thus $\widetilde{\Omega}=\Omega_{1}\cup\Omega_{2}$ and
\begin{equation*}\label{eq2-4}
u(x)=\left\{
\begin{array}{lll}
&0,\quad &\text{if}\quad x\in \Omega_{1};\\
&\delta,\quad &\text{if}\quad x\in\Omega_{2};\\
&u_{0}(x),\quad &\text{if}\quad  x\in \Omega\setminus\widetilde{\Omega}.\\
\end{array}
\right.
\end{equation*}

Moreover, we have
\begin{equation}\label{eq2-5}
\aligned
&\mathcal{E}(u)-\mathcal{E}(u_{0})\\
&=\frac{1}{2}\int\limits_{\mathbb{R}^{N}\times\mathbb{R}^{N}}\frac{|u(x)-u(y)|^{2}-|u_{0}(x)-u_{0}(y)|^{2}}{|x-y|^{N+2s}}dxdy
+\frac{\mu}{2}\int\limits_{\Omega}(|u|^{2}-|u_{0}|^{2})dx\\
&-\int\limits_{\Omega}(G(u)-G(u_{0}))dx\\
&=\frac{1}{2}\int\limits_{\widetilde{\Omega}\times\widetilde{\Omega}}\frac{|u(x)-u(y)|^{2}-|u_{0}(x)-u_{0}(y)|^{2}}{|x-y|^{N+2s}}dxdy
+\frac{\mu}{2}\int\limits_{\Omega\setminus\widetilde{\Omega}}(|u|^{2}-|u_{0}|^{2})dx\\
&+\frac{\mu}{2}\int\limits_{\widetilde{\Omega}}(|u|^{2}-|u_{0}|^{2})dx
-\int\limits_{\Omega\setminus\widetilde{\Omega}}(G(u)-G(u_{0}))dx
-\int\limits_{\widetilde{\Omega}}(G(u)-G(u_{0}))dx\\
&=-\frac{1}{2}\int\limits_{\widetilde{\Omega}\times\widetilde{\Omega}}\frac{|u_{0}(x)-u_{0}(y)|^{2}}{|x-y|^{N+2s}}dxdy
+\int\limits_{\Omega_{2}}\int\limits_{\mathbb{R}^{N}\setminus\widetilde{\Omega}}\frac{|\delta-u_{0}(y)|^{2}-|u_{0}(x)-u_{0}(y)|^{2}}{|x-y|^{N+2s}}dxdy\\
&+\frac{\mu}{2}\int\limits_{\widetilde{\Omega}}(|u|^{2}-|u_{0}|^{2})dx
-\int\limits_{\widetilde{\Omega}}(G(u)-G(u_{0}))dx\\
&=-\frac{1}{2}\int\limits_{\widetilde{\Omega}\times\widetilde{\Omega}}\frac{|u_{0}(x)-u_{0}(y)|^{2}}{|x-y|^{N+2s}}dxdy
+\int\limits_{\Omega_{2}}\int\limits_{\mathbb{R}^{N}\setminus\widetilde{\Omega}}\frac{|\delta|^{2}-|u_{0}(x)|^{2}}{|x-y|^{N+2s}}dxdy\\
&+\frac{\mu}{2}\int\limits_{\widetilde{\Omega}}(|u|^{2}-|u_{0}|^{2})dx
-\int\limits_{\widetilde{\Omega}}(G(u)-G(u_{0}))dx\\
\endaligned
\end{equation}
and
\begin{equation}\label{eq2-6}
\aligned
\int\limits_{\widetilde{\Omega}}(|u|^{2}-|u_{0}|^{2})dx
&=\int\limits_{\Omega_{1}}(|u|^{2}-|u_{0}|^{2})dx
+\int\limits_{\Omega_{2}}(|u|^{2}-|u_{0}|^{2})dx\\
&=-\int\limits_{\Omega_{1}}|u_{0}|^{2}dx
+\int\limits_{\Omega_{2}}(\delta^{2}-|u_{0}|^{2})dx\\
&\leq 0.\\
\endaligned
\end{equation}

On the other hand, from $g(t)=0$ for all $t\leq 0$,
 then we have
\begin{equation}\label{eq2-7}
\aligned
\int\limits_{\Omega_{1}}(G(u)-G(u_{0}))dx
=\int\limits_{\Omega_{1}}(G(0)-0)dx=0\,\,\,\text{for a.e.}\,\,\,x\in \Omega_{1}.
\endaligned
\end{equation}
Moreover, for a.e. $x\in \Omega_{2}$ we have $u(x)=\delta$.
By the mean value theorem and assumption $(g_2)$,
there exists $\theta\in [\delta\,,\,u_{0}]\subset [\delta\,,\,\eta]$
such that
\begin{equation*}
G(u)-G(u_{0})=G(\delta)-G(u_{0})=g(\theta)(\delta-u_{0})\geq 0.
\end{equation*}
This implies that
\begin{equation}\label{eq2-8}
\int\limits_{\Omega_{2}}(G(u)-G(u_{0}))dx
=\int\limits_{\Omega_{2}}g(\theta)(\delta-u_{0})dx\geq 0.
\end{equation}
Then from \eqref{eq2-7} and \eqref{eq2-8}, we get
\begin{equation}\label{eq2-9}
\int\limits_{\widetilde{\Omega}}\Big(G(u)-G(u_{0})\Big)dx\geq 0.
\end{equation}
Therefore, by \eqref{eq2-5}, \eqref{eq2-6} and \eqref{eq2-9}, we obtain
\begin{equation}\label{eq2-10}
\mathcal{E}(u)-\mathcal{E}(u_{0})\leq 0\quad \text{for all }\,\,u\in E_{\eta}.
\end{equation}

On the other hand, since $u_{0}$ is the minimum point of $\mathcal{E}$ over $E_{\eta}$,
 we get
 \begin{equation}\label{eq2-11}
 \mathcal{E}(u)-\mathcal{E}(u_{0})\geq 0\quad \text{for all }\,\,u\in E_{\eta}.
 \end{equation}
Combining \eqref{eq2-10} with \eqref{eq2-11},
we obtain $\mathcal{E}(u)-\mathcal{E}(u_{0})= 0$ for all $u\in E_{\eta}$, which implies
\begin{equation}\label{eq2-11*}
\int\limits_{\widetilde{\Omega}\times\widetilde{\Omega}}\frac{|u_{0}(x)-u_{0}(y)|^{2}}{|x-y|^{N+2s}}dxdy
+2\int\limits_{\widetilde{\Omega}}(G(u)-G(u_{0}))dx
=\mu\int\limits_{\widetilde{\Omega}}(|u|^{2}-|u_{0}|^{2})dx
\end{equation}
In the equality \eqref{eq2-11*},
since the left hand side is non-negative and the right hand side is non-positive,
 every term in \eqref{eq2-11*} should be zero, i.e.,
$$\int\limits_{\widetilde{\Omega}}(|u|^{2}-|u_{0}|^{2})dx=0.$$
So from \eqref{eq2-6} we get
\begin{equation*}\label{eq2-12}
\int\limits_{\Omega_{1}}|u_{0}|^{2}dx
=\int\limits_{\Omega_{2}}(\delta^{2}-|u_{0}|^{2})dx= 0.
\end{equation*}
This implies that $|\Omega_{1}|=0$ and
$|\Omega_{2}|=0$,
 which is a contradiction.
We obtain $u_{0}\in [\,0\,,\,\delta\,]$ for a.e. $x\in \Omega$.

Next, we prove the $u_{0}$ is a weak solution of Eq. \eqref{eq2-1}.
Define the function $J:\,\mathbb{R}\to \mathbb{R}$ by
\begin{equation*}\label{eq2-13}
J(t)=\mathcal{E}(u_{0}+t\,v)\quad \forall v\in \mX.
\end{equation*}
For every $v\in \mX$ we can choose suitable $t$ such that
\begin{equation*}\label{eq2-14}
\aligned
\|u_{0}+t\,v\|_{L^{\infty}(\Omega)}
&\leq \|u_{0}\|_{L^{\infty}(\Omega)}+|t|\cdot\|v\|_{L^{\infty}(\Omega)}\\
&< \|u_{0}\|_{L^{\infty}(\Omega)}+|t|\cdot(\|v\|_{L^{\infty}(\Omega)}+1)\\
&\leq \delta+|t|\cdot(\|v\|_{L^{\infty}(\Omega)}+1)\\
&\leq \eta,
\endaligned
\end{equation*}
which implies that for each
$t\in (-\frac{\eta-\delta}{\|v\|_{L^{\infty}(\Omega)}+1}\,,
\,\frac{\eta-\delta}{\|v\|_{L^{\infty}(\Omega)}+1})$
 we have $u_{0}+t\,v\in E_{\eta}$.
Then $$J(t)=\mathcal{E}(u_{0}+tv)\geq \mathcal{E}(u_{0})=J(0)$$ for all
$t\in (-\frac{\eta-\delta}{\|v\|_{L^{\infty}(\Omega)}+1}\,,
\,\frac{\eta-\delta}{\|v\|_{L^{\infty}(\Omega)}+1})$.
Since $J(t)$ is differentiable at $t=0$ and $J'(0)=0$,
it follows that $\langle\mathcal{E}'(u_{0})\,,\,v\rangle=0$,
and  $u_{0}$ is a weak solution of Eq. \eqref{eq2-1}.
This completes the proof.
\end{proof}

\begin{lemma}\label{l2-3}
Fix $R>0$. For any $\zeta\in (0\,,\,\eta)$, let
\begin{equation}\label{eq2-23}
z_{\zeta}=z(x;\,\zeta,\,R)=\left\{
\begin{array}{lll}
&0,\quad&\text{if}\,\,\, x\in\Omega\setminus B_{2R}, \\
&\zeta,\quad &\text{if}\,\,\, x\in B_{R}, \\
&\frac{\zeta}{R}\,(\,2\,R-|x|\,),\quad &\text{if}\,\,\, x\in B_{2R}\setminus B_{R}, \\
\end{array}\right.
\end{equation}
where $B_{R}\subset \Omega$ be the $N$-dimensional ball with radius $R>0$
and center $0\in \Omega$.
Then $z_{\zeta}\in E_{\eta}$, and there exists $C(N,R,s)>0$ such that
\begin{equation}\label{eq2-23*}
\|z_{\zeta}\|_{\mX}^{2}\leq C(N,R,s)\,\zeta^{2}.
\end{equation}
\end{lemma}

\begin{proof}
The proof follows the lines of \cite{dpv2013}, so we will be brief in details.
Notice that $z_{\zeta}$ is uniformly Lipschitz continuous and vanishes outside $B_{2R}$.
 In particular, we have $z_{\zeta}\in X_{0}^{s}(B_{2R})$ and
\begin{equation}
\begin{aligned}
\label{eq2-30}
&\|z_{\zeta}\|_{\mX}
=\int\limits_{\mathbb{R}^{N}\times\mathbb{R}^{N}}\frac{|z_{\zeta}(x)-z_{\zeta}(y)|^{2}}{|x-y|^{N+2s}}dxdy\\
&=\int\limits_{B_{2R}\times B_{2R}}\frac{|z_{\zeta}(x)-z_{\zeta}(y)|^{2}}{|x-y|^{N+2s}}dxdy+
2\int\limits_{B_{2R}\times (\Omega\setminus B_{2R})}\frac{|z_{\zeta}(x)-z_{\zeta}(y)|^{2}}{|x-y|^{N+2s}}dxdy\\
&+\int\limits_{(\Omega\setminus B_{2R})\times (\Omega\setminus B_{2R})}\frac{|z_{\zeta}(x)-z_{\zeta}(y)|^{2}}{|x-y|^{N+2s}}dxdy\\
&=\int\limits_{B_{R}\times B_{R}}\frac{|z_{\zeta}(x)-z_{\zeta}(y)|^{2}}{|x-y|^{N+2s}}dxdy+
2\int\limits_{B_{R}\times (B_{2R}\setminus B_{R})}\frac{|z_{\zeta}(x)-z_{\zeta}(y)|^{2}}{|x-y|^{N+2s}}dxdy\\
&+\int\limits_{(B_{2R}\setminus B_{R})\times (B_{2R}\setminus B_{R})}\frac{|z_{\zeta}(x)-z_{\zeta}(y)|^{2}}{|x-y|^{N+2s}}dxdy
+2\int\limits_{B_{2R}\times (\Omega\setminus B_{2R})}\frac{|z_{\zeta}(x)-z_{\zeta}(y)|^{2}}{|x-y|^{N+2s}}dxdy\\
&= \int\limits_{B_{R}\times B_{R}}\frac{|\zeta-\zeta|^{2}}{|x-y|^{N+2s}}dxdy
+2\int\limits_{B_{R}}\Big(\int\limits_{B_{2R}\setminus B_{R}}\frac{|\frac{\zeta}{R}(2R-|x|)-\zeta|^{2}}{|x-y|^{N+2s}}dx\Big)dy\\
&+\int\limits_{(B_{2R}\setminus B_{R})\times (B_{2R}\setminus B_{R})}\frac{|\frac{\zeta}{R}(2R-|x|)-\frac{\zeta}{R}(2R-|y|)|^{2}}{|x-y|^{N+2s}}dxdy\\
&+2\int\limits_{B_{2R}\times (\Omega\setminus B_{2R})}\frac{|z_{\zeta}(x)-z_{\zeta}(y)|^{2}}{|x-y|^{N+2s}}dxdy\\
&=2(\frac{\zeta}{R})^{2}\int\limits_{B_{R}}\int\limits_{B_{2R}\setminus B_{R}}\frac{|R-|x||^{2}}{|x-y|^{N+2s}}dxdy
+(\frac{\zeta}{R})^{2}\int\limits_{(B_{2R}\setminus B_{R})\times (B_{2R}\setminus B_{R})}\frac{||x|-|y||^{2}}{|x-y|^{N+2s}}dxdy\\
&+2\int\limits_{B_{2R}\times (\Omega\setminus B_{2R})}\frac{|z_{\zeta}(x)-z_{\zeta}(y)|^{2}}{|x-y|^{N+2s}}dxdy\\
&=:\,2(\frac{\zeta}{R})^{2}\,I+(\frac{\zeta}{R})^{2} \,II+2\,III.
\end{aligned}
\end{equation}

Now, we estimate $I$, $II$ and $III$. Let
\begin{equation*}
\alpha=\left\{
\begin{array}{lll}
 &2s\quad &\text{if}\,\,\,s\in (0\,,\,\frac{1}{2});\\
 &\frac{1}{2}\quad &\text{if}\,\,\,s=\frac{1}{2};\\
 &2s-1\quad &\text{if}\,\,\,s\in (\frac{1}{2}\,,\,1).\\
\end{array}\right.
\end{equation*}
We remark that $\alpha\in (0\,,\,1)$ and by Lemma 13 of \cite{cv2011} we get
\begin{equation}\label{eq2-33}
\int\limits_{B_{R}\times(\mathbb{R}^{N}\setminus B_{R})}\frac{1}{|x-y|^{N+\alpha}}dxdy\leq C(N,\alpha)R^{N-\alpha}.
\end{equation}

For $I$, if $x\in B_{2R}\setminus B_{R}$ and $y\in B_{R}$, we have that
$$||x|-R|=|x|-R\leq |x|-|y|\leq |x-y|$$
and
$$||x|-R|=|x|-R\leq 2R-R=R;$$
therefore, $||x|-R|\leq \min\{|x-y|\,,\,R\}$. Hence
$$||x|-R|^{2}\leq \min\{R^{2},\,\,R|x-y|,\,\,R^{\frac{3}{2}}|x-y|^{\frac{1}{2}}\}.$$
Therefore, for any $x\in B_{2R}\setminus B_{R}$ and $y\in B_{R}$, we get
\begin{equation}\label{eq2-35}
\frac{|R-|x||^{2}}{|x-y|^{N+2s}}
\leq \left\{
\aligned
&\frac{1}{|x-y|^{N+2s}},\quad \text{if}\,\,\,s\in (0\,,\,\frac{1}{2}),\\
&\frac{1}{|x-y|^{N+(2s-\frac{1}{2})}},\quad \text{if}\,\,\,s=\frac{1}{2},\\
&\frac{1}{|x-y|^{N+(2s-1)}},\quad \text{if}\,\,\,s\in (\frac{1}{2}\,,\,1),\\
\endaligned\right.
\end{equation}
that is,
$$
\frac{|R-|x||^{2}}{|x-y|^{N+2s}}\leq\frac{1}{|x-y|^{N+\alpha}}\,\,\,\text{for any}\,\,s\in (0\,,\,1).
$$
Then
\begin{equation}\label{eq2-31}
\aligned
I:
&=\int\limits_{B_{R}}\Big(\int\limits_{B_{2R}\setminus B_{R}}\frac{|R-|x||^{2}}{|x-y|^{N+2s}}dx\Big)dy
\leq \int\limits_{B_{R}}\Big(\int\limits_{B_{2R}\setminus B_{R}}\frac{||x|-|y||^{2}}{|x-y|^{N+2s}}dx\Big)dy\\
&\leq \int\limits_{B_{R}}\Big(\int\limits_{B_{2R}\setminus B_{R}}\frac{|x-y|^{2}}{|x-y|^{N+2s}}dx\Big)dy
\leq \int\limits_{B_{R}}\Big(\int\limits_{\mathbb{R}^{N}\setminus B_{R}}\frac{1}{|x-y|^{N+\alpha}}dx\Big)dy\\
&\leq C_{1}(N,R,s).
\endaligned
\end{equation}

Now, we estimate $II$. If $x,\,y\in B_{2R}\setminus B_{R}$, we have that $|x-y|\leq 4R$.
Thus, we make the substitution $\theta:=x-y$ and let $\rho=|\theta|$ in the following computation
\begin{equation}\label{eq2-32}
\aligned
II:
&=\int\limits_{(B_{2R}\setminus B_{R})\times (B_{2R}\setminus B_{R})}\frac{||x|-|y||^{2}}{|x-y|^{N+2s}}dxdy
\leq \int\limits_{(B_{2R}\setminus B_{R})\times (B_{2R}\setminus B_{R})}\frac{|x-y|^{2}}{|x-y|^{N+2s}}dxdy\\
&\leq \int\limits_{(B_{2R}\setminus B_{R})\times (B_{2R}\setminus B_{R})}\frac{1}{|x-y|^{N+2s-2}}dxdy
\leq \int\limits_{(B_{2R}\setminus B_{R})}\Big(\int\limits_{B_{4R}}|\theta|^{2-N-2s}d\theta\Big)dy\\
&\leq C(N,R)\int\limits_{0}^{4R}\varrho^{2-N-2s}\varrho ^{N-1}d\varrho
\leq C(N,R)\frac{(4R)^{2-2s}}{2-2s}=:C_{2}(N,R,s),
\endaligned
\end{equation}
for a suitable $C_{2}(N,R,s)>0$.

For $III$, if $x\in \mathbb{R}^{N}\setminus B_{2R}$ and  $y\in B_{R}$, we get
\begin{equation}\label{eq2-34-1}
|x-y|\geq |x|-|y|\geq |x|-\frac{x}{2R}|y|=\frac{|x|}{2R}(2R-|y|)\geq \frac{|x|}{2R}(2R-R)=\frac{|x|}{2}>0.
\end{equation}
Moreover, similar as the estimate $II$, if $y\in B_{2R}\setminus B_{R}$ and $x\in \mathbb{R}^{N}\setminus B_{2R}$,
there exists $C_{3}(N,R,s)>0$ such that
\begin{equation*}
\aligned
(\frac{\zeta}{R})^{2}\int\limits_{B_{2R}\setminus B_{R}}\Big(\int\limits_{\Omega\setminus B_{2R}}\frac{|x-y|^{2}}{|x-y|^{N+2s}}dx\Big)dy
&\leq (\frac{\zeta}{R})^{2}\int\limits_{B_{2R}}\Big(\int\limits_{\mathbb{R}^{N}\setminus B_{2R}}\frac{1}{|x-y|^{N+2s-2}}dx\Big)dy\\
&\leq C_{3}(N,R,s).\\
\endaligned
\end{equation*}
Then we have
\begin{equation}\label{eq2-34}
\aligned
III:
&=\int\limits_{B_{2R}\times (\Omega\setminus B_{2R})}\frac{|z_{\zeta}(x)-z_{\zeta}(y)|^{2}}{|x-y|^{N+2s}}dxdy\\
&=\int\limits_{B_{R}}\Big(\int\limits_{ \Omega\setminus B_{2R}}\frac{|\zeta|^{2}}{|x-y|^{N+2s}}dx\Big)dy
+\int\limits_{B_{2R}\setminus B_{R}}\Big(\int\limits_{\Omega\setminus B_{2R}}\frac{|0-\frac{\zeta}{R}(2R-|y|)|^{2}}{|x-y|^{N+2s}}dx\Big)dy\\
&\leq \zeta^{2}\int\limits_{B_{R}}\Big(\int\limits_{ (\Omega\setminus B_{2R})}\frac{1}{|x-y|^{N+2s}}dx\Big)dy
+(\frac{\zeta}{R})^{2}\int\limits_{B_{2R}\setminus B_{R}}\Big(\int\limits_{\Omega\setminus B_{2R}}\frac{|x-y|^{2}}{|x-y|^{N+2s}}dx\Big)dy\\
&\leq \zeta^{2}\int\limits_{B_{R}}\Big(\int\limits_{ \mathbb{R}^{N}\setminus B_{2R}}\Big(\frac{2}{|x|}\Big)^{N+2s}dx\Big)dy
+(\frac{\zeta}{R})^{2}\int\limits_{B_{2R}}\Big(\int\limits_{\mathbb{R}^{N}\setminus B_{2R}}\frac{1}{|x-y|^{N+2s-2}}dx\Big)dy\\
&\leq \zeta^{2}\int\limits_{B_{R}}\int\limits_{2R}^{+\infty}\frac{1}{\varrho^{2s+1}}d\varrho+(\frac{\zeta}{R})^{2}C(N,R,s)\\
&\leq C_{4}(N,R,s)\zeta^{2}.
\endaligned
\end{equation}
Thus, making use of \eqref{eq2-31}, \eqref{eq2-32}, \eqref{eq2-34} and \eqref{eq2-30} we conclude that
there exists $C(N,R,s)>0$ such that
$\|z_{\zeta}\|_{\mX}^{2}\leq C(N,R,s)\,\zeta^{2}$. This completes the proof.
\end{proof}

Now, we consider the oscillatory behavior of $g$ near the origin, and assume that
\begin{itemize}
\item[($G_1$)] $g(0)=0$ and there exists $s_{0}>0$
such that $\sup\limits_{t\in [0\,,\,s_{0}]}|g(t)|<\infty$;
\item[($G_2$)]$-\infty<\liminf\limits_{t\to 0^{+}}\frac{G(t)}{t^2}
\leq\limsup\limits_{t\to 0^{+}}\frac{G(t)}{t^2}=+\infty$
uniformly for a.e. $x\in \Omega$, where $G(t)=\displaystyle\int_{0}^{t}g(\tau)d\tau$;
\item[($G_3$)] There are two sequences $\{\delta_{k}\}_{k\in \mathbb{N}}$,
  $\{\eta_{k}\}_{k\in \mathbb{N}}$ with
$0<\eta_{k+1}<\delta_{k}<\eta_{k}$ and $\lim\limits_{k\to \infty}\eta_{k}=0$ such that
$g(u)\leq 0$ for all $u\in [\delta_{k}\,,\,\eta_{k}]$.
\end{itemize}

We have the following result.
 \begin{theorem}\label{th2-2}
Assume that $(G_1)-(G_3)$ hold. Then there exist infinitely many positive solutions
 $\{u_{k}\}_{k\in \mathbb{N}}\subset \mX$
of Eq. \eqref{eq2-1} satisfying
\begin{equation}\label{eq2*}
\lim\limits_{k\to \infty}\|u_{k}\|_{\mX}
 =\lim\limits_{k\to \infty}\|u_{k}\|_{L^{\infty}(\Omega)}=0.
 \end{equation}
\end{theorem}

\begin{proof}
Define the truncation functions $g_{k}:\,\mathbb{R}^{+}\to \mathbb{R}$
as follows
\begin{equation}\label{eq2-16}
g_{k}(t)=g(\tau_{\eta_{k}}(t))\quad k\in \mathbb{N},
\end{equation}
where $\tau_{\eta}:\,[0\,,\,+\infty)\to \mathbb{R}$ is given by
$\tau_{\eta}(t)=\min\{\eta\,,\,t\}$, $t\geq 0$.

Let us consider the problem
\begin{equation}\label{eq2-1-g}
\left\{\begin{array}{lll}
&(-\Delta)^{s}u+\mu u=g_{k}(u),\,\,u\geq 0\quad &\text{in}\,\,\Omega,\\
&u=0\quad &\text{in}\,\,\mathbb{R}^{N}\setminus\Omega,\\
\end{array}\right.
\end{equation}
The associated energy functional with respect to \eqref{eq2-1-g} is
\begin{equation}\label{eq2-17}
\mathcal{E}_{k}(w)=\frac{1}{2}\int\limits_{\mathbb{R}^{N}\times\mathbb{R}^{N}}\frac{| u(x)-u(y)|^{2}}{|x-y|^{N+2s}}dxdy+
\frac{\mu}{2}\int\limits_{\Omega}|u|^{2}dx-\int\limits_{\Omega}G_{k}(u)dx,
\end{equation}
where $G_{k}(u)=\displaystyle\int_{0}^{u}g_{k}(t)dt$.

Without any loss of generality, we may assume that the sequences
$\{\delta_{k}\}_{k\in \mathbb{N}}$and $\{\eta_{k}\}_{k\in \mathbb{N}}$
satisfy $0<\delta_{k}<s_{0}$, $0<\eta_{k}<s_{0}$
for all $k\in \mathbb{N}$, where $s_{0}>0$ is given in $(G_1)$.
From the definition of $g_{k}$ and assumptions $(G_1)$ and $(G_3)$,
it is easy to see that the function $g_{k}(t)$
satisfies the assumptions of Theorem \ref{th2-1} for all $t\in[0\,,\,s_{0}]$.
Therefore,  for every $k\in \mathbb{N}$,
there exists $u_{k}\in E_{\eta_{k}}$ satisfying
\begin{itemize}
\item[(i)] $u_{k}$ is the minimum point of the function $\mathcal{E}_{k}$ over $E_{\eta_{k}}$;
\item[(ii)] $u_{k}$ is a weak solutions of problem \eqref{eq2-1-g} and
           $u_{k}\in[0\,,\,\delta_{k}]$ for a.e. $x\in \Omega$.
\end{itemize}
Moreover, due to \eqref{eq2-16} and $u_{k}\in[0\,,\,\delta_{k}]\subset[0\,,\,\eta_{k}]$,
we get $$g_{k}(u_{k})=g(u_{k})\quad \text{for all}\,\,\, u_{k}\in [0\,,\,\delta_{k}].$$
So $u_{k}$ is a weak solutions not only for Eq. \eqref{eq2-1-g}  but also for the original problem \eqref{eq2-1}.

Next, we  prove that there are many distinct elements in the sequence $\{u_{k}\}_{k\in \mathbb{N}}$,
that is, Eq. \eqref{eq2-1} has infinitely many positive solutions.

By $(G_2)$ there exist $l_{k}>0$ and $\xi\in (0\,,\,\eta_{1})$ such that
\begin{equation}\label{eq2-20}
 G(u)\geq -l_{k}\,u^{2}\quad \text{for all}\,\,u\in (0\,,\,\xi),
\end{equation}
and  there exist $L_{k}>0$ and
 a subsequence of $\{u_{k}\}_{k\in \mathbb{N}}$,
 still denoted by $\{u_{k}\}_{k\in \mathbb{N}}$,
 with $0<u_{k}<\min\{\xi\,,\,\delta_{k}\}$,
 such that
\begin{equation}\label{eq2-22}
 G(u_{k})> L_{k}\,u_{k}^{2}\quad \text{for every}\,\,k\in \mathbb{N}.
\end{equation}

Let $k\in \mathbb{N}$ be fixed and
$z_{k}\in \mX$ be as the function from \eqref{eq2-23}  corresponding
 to the value $u_{k}>0$. It is clear that $z_{k}\in E_{\eta_{k}}$ for every $k\in \mathbb{N}$.

 Now we calculate the function value of $\mathcal{E}_{k}(z_{k})$.
 On account of \eqref{eq2-20}, \eqref{eq2-22} and \eqref{eq2-23*}, we have
\begin{equation}\label{eq2-24}
\aligned
\mathcal{E}_{k}(z_{k})
&=\frac{1}{2}\int\limits_{\mathbb{R}^{N}\times\mathbb{R}^{N}}\frac{|z_{k}(x)-z_{k}(y)|^{2}}{|x-y|^{N+2s}}dxdy+
\frac{\mu}{2}\int\limits_{\Omega}|z_{k}|^{2}dx
-\int\limits_{\Omega}G_{k}(z_{k})dx\\
&\leq \frac{C(N,R,s)}{2}\,u_{k}^{2}+\frac{\mu}{2}|\Omega|\,u_{k}^{2}-\int\limits_{B_{R}}G(z_{k})dx
-\int\limits_{B_{2R}\setminus B_{R}}G(z_{k})dx\\
&\leq \Big(\frac{C(N,R,s)}{2}+\frac{\mu}{2}|\Omega|-L_{k}|B_{R}|+l_{k}|B_{R}|\Big)\,u_{k}^{2},\\
\endaligned
\end{equation}
where $|B_{R}|$ denotes the Lebesgue measure of $B_{R}$.
From \eqref{eq2-24}, we can take $L_{k}>0$ large enough such that
\begin{equation*}\label{eq2-21}
\frac{C(N,R,s)}{2}+\frac{\mu}{2}|\Omega|-L_{k}|B_{R}|+l_{k}|B_{R}|<0,
\end{equation*}
which implies that $\mathcal{E}_{k}(z_{k})<0$ for all $k\in \mathbb{N}$.
Consequently,
\begin{equation}\label{eq2-25}
\mathcal{E}_{k}(u_{k})=\min\limits_{E_{\eta_{k}}}\mathcal{E}_{k}\leq
\mathcal{E}_{k}(z_{k})<0.
\end{equation}

On the other hand, for every $k$, using $(G_{1})$, there exists $M>0$ such that
\begin{equation}\label{eq2-26}
\aligned
\mathcal{E}_{k}(u_{k})
&\geq-\int\limits_{\Omega}G_{k}(u_{k})dx
\geq -|\sup\limits_{t\in [0\,,\,s_{0}]}g_{k}(t)|\,|\Omega|\,\delta_{k}\\
&\geq -M|\Omega|\,\delta_{k}.\\
\endaligned
\end{equation}
Taking the limits $k\to \infty$ in \eqref{eq2-26} and combining with \eqref{eq2-25}, $\delta_{k}\to 0$,
we have
\begin{equation}\label{eq2-19}
\lim\limits_{k\to \infty}\mathcal{E}_{k}(u_{k})=0.
\end{equation}

Moreover, from $u_{k}\leq \delta_{k}<\eta_{k}<\delta_{k-1}<\cdot\cdot\cdot<\eta_{1}$, we obtain
$$\tau_{\eta_{1}}(u_{k})=\min\{\eta_{1}\,,\,u_{k}\}=u_{k}$$
and
$$\tau_{\eta_{k}}(u_{k})=\min\{\eta_{k}\,,\,u_{k}\}=u_{k}$$
for all $k\in \mathbb{N}$. Thus
\begin{equation*}
g_{k}(u_{k})=g(\tau_{\eta_{k}}(u_{k}))=g(u_{k})=g(\tau_{\eta_{1}}(u_{k}))=g_{1}(u_{k})
\quad \forall k\in \mathbb{N},
\end{equation*}
which implies that
\begin{equation}\label{eq2-27}
\mathcal{E}_{k}(u_{k})=\mathcal{E}_{1}(u_{k})\quad \text{for all}\,\,k\in \mathbb{N}.
\end{equation}
Combing \eqref{eq2-25} with \eqref{eq2-19} and \eqref{eq2-27},
we have
$$\lim\limits_{k\to \infty}\mathcal{E}_{1}(u_{k})=0\,\,\,
\text{and}\,\,\,\mathcal{E}_{1}(u_{k})<0\,\,\,\text{for all}\,\,\,k\in \mathbb{N},$$
which yield that the sequence $\{u_{k}\}_{k\in \mathbb{N}}$ contains infinitely many distinct  elements
and $u_{k}\in(0\,,\,\delta_{k})$ for every $k\in \mathbb{N}$.

Finally, we remains to prove \eqref{eq2*}. At first,
it follows from the result (ii) that
 $\|u_{k}\|_{L^{\infty}(\Omega)}\leq \delta_{k}$
for all $k\in \mathbb{N}$, combined with $\lim\limits_{k\to \infty}\delta_{k}=0$,
 and then we get
$$\lim\limits_{k\to \infty}\|u_{k}\|_{L^{\infty}(\Omega)}=0.$$
For the left, from $(G_1)$, $\lim\limits_{k\to \infty}\delta_{k}=0$
and $u_{k}$ is a nontrivial critical point of functional $\mathcal{E}_{1}$, we have
\begin{equation*}\label{eq2-28}
\aligned
\|u_{k}\|^{2}_{\mX}
&\leq \|u_{k}\|^{2}_{\mX}+\mu \int\limits_{\Omega}|u_{k}|^{2}dx=
\int\limits_{\Omega}g(u_{k})u_{k}dx\\
&\leq|\sup g(u_{k})|\,|\Omega|\,\delta_{k}\leq M \delta_{k}\\
&\to 0\quad (\text{as}\,\,k\to \infty),
\endaligned\end{equation*}
which concludes the proof of \eqref{eq2*}. This completes the proof of Theorem \ref{th2-2}.
\end{proof}

Similar as in Theorem \ref{th2-2},
we consider the oscillatory behavior of $g$ at infinity. In this case, we assume
the continuous function $g:\,[0\,,\,+\infty)\to \mathbb{R}$ fulfills
\begin{itemize}
\item[$(G'_1)$] $g(0)=0$ and for every $s>0$, $\sup\limits_{t\in [0,s]}|g(t)|<\infty$;
\item[$(G'_2)$] $-\infty<\liminf\limits_{t\to \infty}\frac{G(t)}{t^2}
                \leq\limsup\limits_{t\to \infty}\frac{G(t)}{t^2}=+\infty$
                uniformly for a.e. $x\in\Omega$,
                where $G(t)=\displaystyle{\int}_{0}^{t}g(\tau)d\tau$;
\item[$(G'_3)$] There exist two sequences $\{\delta_{k}\}_{k\in \mathbb{N}}\subset (0\,,\,+\infty)$
               and $\{\eta_{k}\}_{k\in \mathbb{N}}\subset (0\,,\,\infty)$ with $0<\delta_{k}<\eta_{k}<\delta_{k+1}$,
               $\lim\limits_{k\to \infty}\delta_{k}=+\infty$,
               such that $g(t)\leq 0$ for every $t\in [\,\delta_{k}\,,\,\eta_{k}\,]$.
\end{itemize}
We have the following result.
\begin{theorem}\label{th2-3}
Assume that $(G'_1)-(G'_3)$ hold. Then there exist infinitely many positive solutions
$\{u_{k}\}_{k\in \mathbb{N}}\subset \mX$
of Eq. \eqref{eq2-1} such that
\begin{equation}\label{eq4-1}
\lim\limits_{k\to \infty}\|u_{k}\|_{L^{\infty}(\Omega)}=\infty.
\end{equation}
\end{theorem}

\begin{proof}
We consider the minimum problem
$$c_{k}:=\inf\limits_{E_{\eta_{k}}}\mathcal{E}_{k},$$
where the functional $\mathcal{E}_{k}$ is given in \eqref{eq2-17} and
 the truncation function $g_{k}:\,(0\,,\,+\infty)\to \mathbb{R}$ defined by
\begin{equation*}\label{eq4-2}
g_{k}(t)=g(\tau_{\eta_{k}}(t))\quad t>0,
\end{equation*}
where $\tau_{\eta}(t)=\min\{\,\eta\,,\,t\,\}$, $\eta>0$.

On account of the definition of $g_{k}$ and hypotheses $(G'_{1})$, $(G'_{3})$,
it is easy to see that $g_{k}$ fulfills the assumptions of Theorem \ref{th2-1}.
Therefore, for every $k\in \mathbb{N}$,
 there is a $u_{k}\in E_{\eta_{k}}$ such that
\begin{itemize}
\item[(i)] $u_{k}$ is the minimum point of the functional $\mathcal{E}_{k}$ on $E_{\eta_{k}}$;
\item[(ii)] $u_{k}$ is a weak solutions of Eq. \eqref{eq2-1} with $g=g_{k}$, and
$u_{k}\in [\,0\,,\,\delta_{k}\,]$ for a.e. $x\in \Omega$,
\end{itemize}

Thanks to the above results (i) and (ii), we have that
$$\tau_{\eta_{k}}(u_{k})=\min\{\eta_{k}\,,\,u_{k}\}=u_{k},\quad (\text{since}\,\,0<\delta_{k}<\eta_{k})$$
and
$$g_{k}(u_{k})=g(\tau_{\eta_{k}}(u_{k}))=g(u_{k})\quad \text{for every}\,\,k\in \mathbb{N},$$
that is, $\{u_{k}\}_{k\in \mathbb{N}}$ is a solution sequence for the Eq. \eqref{eq2-1}.

Next, we claim that there are infinitely many solutions for problem \eqref{eq2-1}.
To this end, it is enough to show that
\begin{equation}\label{eq4-6}
\lim\limits_{k\to \infty}c_{k}=-\infty.
\end{equation}
Arguing by contradiction, we assume that in the sequence $\{u_{k}\}_{k\in \mathbb{N}}$
there are only finitely distinct elements,
i.e., $\{u_{1},u_{2},\cdot\cdot\cdot,u_{k_{0}}\}$ for some $k_{0}\in \mathbb{N}$.
Then for every $k\in \{1,2,\cdot\cdot\cdot,k_{0}\}$,
we have
$$u_{k}\leq \delta_{k}<\eta_{k}\leq \delta_{k_{0}}<\eta_{k_{0}}.$$
So, for every $k\in \{1,2,\cdot\cdot\cdot,k_{0}\}$, we get
$$g_{k}(u_{k})=g(\tau_{\eta_{k}}(u_{k}))=g(\min\{\eta_{k}\,,\,u_{k}\})=g(u_{k})
=g(\eta_{k_{0}}(u_{k}))=g_{k_{0}}(u_{k}),$$
 which implies $G_{k}(u_{k})=G_{k_{0}}(u_{k})$ for all $k\in \{1,2,\cdot\cdot\cdot,k_{0}\}$.
Therefore, if follows that
\begin{equation*}
\mathcal{E}_{k}(u_{k})=\mathcal{E}_{k_{0}}(u_{k})\quad \text{for every $k\in \{1,2,\cdot\cdot,k_{0}\}$},
\end{equation*}
that is, the sequence $\{\mathcal{E}_{k}(u_{k})\}$
reduces to at most the finite set
$\{\mathcal{E}_{k_{0}}(u_{1}),\mathcal{E}_{k_{0}}(u_{2}),\cdot\cdot\cdot,\mathcal{E}_{k_{0}}(u_{k_{0}})\}$,
which contradicts \eqref{eq4-6}.

Now, we prove \eqref{eq4-6}.
The left hand side of  $(G'_2)$ implies that there exist $l_{\infty}>0$
and $\xi_{0}>0$ such that
\begin{equation}\label{eq4-7}
 G(u)\geq -l_{\infty}u^{2}\quad \text{for all}\,\,u>\xi_{0}.
\end{equation}
Moreover, the right hand side of  $(G'_2)$ implies
there exist $L_{\infty}>0$ and a  strictly increasing sequence
 $\{u_{k}\}_{k\in \mathbb{N}}\subset (0\,,\,+\infty)$
such that $$\lim\limits_{k\to\infty}u_{k}=+\infty,$$ and
\begin{equation}\label{eq4-9}
 G(u_{k})>L_{\infty}\,u_{k}\quad\text{for all}\quad k\in \mathbb{N}.
\end{equation}

 Since $\lim\limits_{k\to \infty}\delta_{k}=+\infty$,
 up to a subsequence if necessary,
we can choose a subsequence of $\{\delta_{k}\}$, still denotes by $\{\delta_{k}\}$, such that
 $u_{k}\leq \delta_{k}$ for all $k\in \mathbb{N}$.
Let the function $\tilde{z}_{k}\in \mX$ be as \eqref{eq2-23}  corresponding
 to the value $u_{k}>0$. It is clear that $\tilde{z}_{k}\in E_{\eta_{k}}$.
Due to \eqref{eq2-23*}, \eqref{eq4-7} and \eqref{eq4-9}, we have
\begin{equation}\label{eq4-10}
\aligned
\mathcal{E}_{k}(\tilde{z}_{k})
&=\frac{1}{2}\int\limits_{\mathbb{R}^{N}\times\mathbb{R}^{N}}\frac{|\tilde{z}_{k}(x)-\tilde{z}_{k}(y)|^{2}}{|x-y|^{N+2s}}dxdy
+\frac{\mu}{2}\int\limits_{\Omega}|\tilde{z}_{k}|^{2}dx-\int\limits_{\Omega}G_{k}(\tilde{z}_{k})dx\\
&\leq \frac{C(N,R,s)}{2}\,u_{k}^{2}+\frac{\mu}{2}|\Omega|\,u_{k}^{2}-\int\limits_{B_{R}}G(\tilde{z}_{k})dx
-\int\limits_{(B_{2R}\setminus B_{R})\cap \{u_{k}>\xi_{0}\}}G(\tilde{z}_{k})dx\\
& -\int\limits_{(B_{2R}\setminus B_{R})\cap \{u_{k}\leq \xi_{0}\}}G(\tilde{z}_{k})dx\\
&\leq \Big(\frac{C(N,R,s)}{2}+\frac{\mu}{2}|\Omega|-L_{k}|B_{R}|+l_{k}|B_{R}|\Big)\,u_{k}^{2}+M\,|\Omega|\,\xi_{0},\\
\endaligned
\end{equation}
where $M:=\sup_{t\in [0\,,\,\xi_{0}]}|g(t)|$ and $|B_{R}|$ denotes the Lebesgue measure of the set $B_{R}$.

If we take $L_{\infty}>0$ large enough such that
\begin{equation}\label{eq4-8}
\frac{C(N,R,s)}{2}+\frac{\mu}{2}|\Omega|-L_{k}|B_{R}|+l_{k}|B_{R}|<0,
\end{equation}
then from \eqref{eq4-10}, \eqref{eq4-8} and $\lim\limits_{k\to \infty}u_{k}=+\infty$,
we get
\begin{equation*}\label{eq4-11}
\lim\limits_{k\to \infty}\mathcal{E}_{k}(\tilde{z}_{k})\to -\infty.
\end{equation*}
Therefore, from the fact that $u_{k}$ is the minimum point of the functional $\mathcal{E}_{k}$,
we have
$$
c_{k}:=\mathcal{E}_{k}(u_{k})=\min\limits_{E_{\eta_{k}}}\mathcal{E}_{k}\leq\mathcal{E}_{k}(\tilde{z}_{k}),
$$
for all $k\in \mathbb{N}$, which implies \eqref{eq4-6}.

Moreover, for every $k\in \mathbb{N}$, it follows from $u_{k}\in[0\,,\,\delta_{k}]\subset [0\,,\,\eta_{k}]$
that $g_{k}(u_{k})=g(u_{k})$,
which implies that Eq. \eqref{eq2-1} has infinitely many positive solutions.

Now, let us prove  \eqref{eq4-1}.
Arguing by contradiction,
we assume that there exists $0<M<+\infty$ such that
 $$\|u_{k}\|_{L^{\infty}(\Omega)}\leq M.$$
By mean value theorem and H\"{o}lder inequality, we note that
\begin{equation*}\label{eq4-15}
\aligned
\mathcal{E}_{k}(u_{k})
&=\frac{1}{2}\int\limits_{\mathbb{R}^{N}\times\mathbb{R}^{N}}\frac{| u_{k}(x)-u_{k}(y)|^{2}}{|x-y|^{N+2s}}dxdy
+\frac{\mu}{2}\int\limits_{\Omega}|u_{k}|^{2}dx
-\int\limits_{\Omega}G(u_{k})\\
&\geq- \int\limits_{\Omega}G(u_{k})\\
&\geq-\max\limits_{u_{k}\in [0,M]}|g(u_{k})|\cdot |\Omega|\cdot M.
\endaligned
\end{equation*}
As a consequence, $c_{k}\geq -\max\limits_{u_{k}\in [0,M]}|g(u_{k})| |\Omega|M$, which contradicts \eqref{eq4-6}.
This concludes the proof of Theorem \ref{th2-3}.
\end{proof}

\section{Proof of Theorems \ref{th1-1} and \ref{th1-2}}

In this section, we will prove the results of Theorem \ref{th1-1} and \ref{th1-2}. Note that $p>1$
may be critical or even supercritical in Theorem \ref{th1-1} (b).
From the assumption, Theorem \ref{th1-1} roughly says that the term defined by
$u\mapsto u^{p}$ $(u>0)$ does not affect the number of solutions of \eqref{eq1-1}
whenever $p>1$. This is also the case for certain values of $\lambda\in \mathbb{R}$
when $p=1$.

 On the other hand, for the case $0<p<1$, the term $u\mapsto u^{p}$ $(u>0)$ may
 compete with the function $f$ near the origin such that the number of the solutions
 becomes finite for the values of $\lambda\in \mathbb{R}$.

\vspace{5mm}
First we give the proof in case that the nonlinear term $f$ is oscillatory near the origin.
\vspace{3mm}

{\bf Proof of Theorem \ref{th1-1}. }
(a) Case $p=1$.
Let
\begin{equation}\label{eq3-1}
\mu=\tilde{\lambda}_{0}-\lambda \quad \text{and}\quad g(u)=\tilde{\lambda}_{0}u+f(u),\,\,\forall u\in \mX,\,\,u\geq 0,
\end{equation}
where $\tilde{\lambda}_{0}\in (\lambda_{0}\,,\,-l_{0})$.
Then $\mu=\tilde{\lambda}-\lambda>\lambda_{0}-\lambda>0$. By $f(0)=0$, we have $g(0)=0$.
Moreover, since
 \begin{equation*}\label{eq3-2}
 \frac{G(u)}{u^{2}}=\frac{\tilde{\lambda}_{0}}{2}+\frac{F(u)}{u^{2}},\quad
 \forall u>0,
 \end{equation*}
  we have
 $$\liminf\limits_{u\to 0^{+}}\frac{G(u)}{u^2}=\frac{\tilde{\lambda}_{0}}{2}+\liminf\limits_{u\to 0^{+}}\frac{F(u)}{u^2}$$
  and
 $$\limsup\limits_{u\to 0^{+}}\frac{G(u)}{u^2}=\frac{\tilde{\lambda}_{0}}{2}+\limsup\limits_{u\to 0^{+}}\frac{F(u)}{u^2}.$$
Then the condition $(f_0)$ implies $(G_2)$.
Finally,
since $l_{0}=\liminf\limits_{u\to 0^{+}}\frac{f(u)}{u}<-\tilde{\lambda}_{0}$,
there exists a positive sequence $\{u_{k}\}_{k\in\mathbb{N}}$
 such that $u_{k}\to 0$ as $k\to \infty$ and
$$\frac{f(u_{k})}{u_{k}}<-\tilde{\lambda}_{0}\quad
\text{for all} \quad k\in \mathbb{N}.$$
Then, by continuity of $f$,
we may choose two sequences $\{\delta_{k}\}_{k\in \mathbb{N}}\subset(0\,,\,\infty)$ and
$\{\eta_{k}\}_{k\in \mathbb{N}}\subset(0\,,\,\infty)$ such that
\begin{equation*}\label{eq3-3}
0<\eta_{k+1}<\delta_{k}<u_{k}<\eta_{k},\quad  \lim\limits_{k\to \infty}\eta_{k}=0,
\end{equation*}
and $$\tilde{\lambda}_{0}u+f(u)\leq 0\,\,\,\text{for all} \,\,\,u\in [\delta_{k}\,,\,\eta_{k}].$$
So $g(u)\leq 0$ for every $u\in [\delta_{k}\,,\,\eta_{k}]$, that is,
the assumption $(G_3)$ of Theorem \ref{th2-2} holds too.
It remains to apply Theorem \ref{th2-2},
observing that Eq. \eqref{eq1-1} is equivalent to problem \eqref{eq2-1} via the choice \eqref{eq3-1}.

(b) Case $p>1$.
Let $\lambda_{0}\in (0\,,\,-l_{0})$ and choose
\begin{equation}\label{eq3-4}
\mu=\lambda_{0}, \quad g(u)=\lambda u^{p}+\lambda_{0}u+f(u),
\quad \forall u\in \mX,\,\, \,u\geq 0.
\end{equation}
Then $\mu>0$ and $g(0)=0$, which yields that $(G_1)$ holds.

  Moreover, since $p>1$ and
\begin{equation*}\label{eq3-5}
\frac{G(u)}{u^2}=\frac{1}{p+1}\lambda \,u^{p-1}+\frac{\lambda_{0}}{2}+\frac{F(u)}{u^2},
\end{equation*}
we have
$$\liminf\limits_{u\to 0^{+}}\frac{G(u)}{u^2}=\frac{\lambda_{0}}{2}+\liminf\limits_{u\to 0^{+}}\frac{F(u)}{u^2}$$
and
$$\limsup\limits_{u\to 0^{+}}\frac{G(u)}{u^2}=\frac{\lambda_{0}}{2}+\limsup\limits_{u\to 0^{+}}\frac{F(u)}{u^2}.$$
 Then the hypothesis $(f_1)$ implies ($G_2$) holds.

Now we prove $(G_3)$ holds.
For every $u\geq 0$,
\begin{equation*}\label{eq3-6}
g(u)\leq |\lambda|\,u^{p}+\lambda_{0}\,u+f(u),
\end{equation*}
 implies
$$\liminf\limits_{u\to 0^{+}}\frac{g(u)}{u}\leq \lambda_{0}+\liminf\limits_{u\to 0^{+}}\frac{f(u)}{u}=\lambda_{0}+l_{0}<0.$$
In particular, we may fix a positive sequence $\{u_{k}\}_{k\in \mathbb{N}}$
converging to 0 as $k\to \infty$ such that $g(u_{k})\leq 0$ for all $k\in \mathbb{N}$.
Then by the continuity of $g$, we can choose two positive sequences $\{\delta_{k}\}_{k\in \mathbb{N}}$,
$\{\eta_{k}\}_{k\in \mathbb{N}}$ with $0<\eta_{k+1}<\delta_{k}<u_{k}<\eta_{k}$ such that
  $$\eta_{k}\to 0\,\,\,\text{as}\,\,\,k\to \infty,$$
and
\begin{equation*}\label{eq3-7}
g(u)\leq 0\quad\text{for all }\,\,u\in [\delta_{k}\,,\,\eta_{k}].
\end{equation*}
Therefore,  the hypothesis $(G_3)$ holds for every $k\in \mathbb{N}$.
Now we can  apply Theorem \ref{th2-2} since
 Eq. \eqref{eq1-1} is equivalent to problem \eqref{eq2-1} through the choice \eqref{eq3-4}.
 This completes the proof of Theorem \ref{th1-1}.

\vspace{10mm}

{\bf Proof of Theorem \ref{th1-2}.}
The proof is divided into three steps.

{\bf Step 1.} Let $\lambda_{0}\in (0\,,\,-l_{0})$ and choose
\begin{equation}\label{eq3-9}
\mu=\lambda_{0},\,\,\,\text{and}\,\,\, g(u)=\lambda \,u^{p}+\lambda_{0}\,u+f(u),\quad \forall u\in \mX,\,\,u\geq 0.
\end{equation}

Define the function $\widetilde{g}:\,\mathbb{R}\times[0\,,\,\infty)\to \mathbb{R}$ by
 \begin{equation*}\label{eq3-9*}
\widetilde{g}(\lambda\,,\,u)=|\lambda|\,u^{p}+\lambda_{0}\,u+f(u).
\end{equation*}
It is easy to check that $\widetilde{g}\in C^{1}$ with respect to $u$ and $\lambda$,
 and $g(u)\leq \widetilde{g}(\lambda\,,\,u)$ for every $\lambda\in \mathbb{R}$ and $u\in [0\,,\,+\infty)$

Since $l_{0}:=\lim\inf\limits_{u\to 0^{+}}\frac{f(u)}{u}<0$ and $\lambda_{0}\in (0\,,\,-l_{0})$,
there exists a sequence $\{u_{k}\}_{k\in \mathbb{N}}\subset (0\,,\,\infty)$
converging to 0 as $k\to \infty$ such that
 $\frac{f(u_{k})}{u_{k}}<-\lambda_{0}$,
that is, $\lambda_{0}\,u_{k}+f(u_{k})<0$ for all $u_{k}>0$.
 Therefore, $\widetilde{g}(0\,,\,u_{k})<0$ for all $u_{k}>0$.
 By the continuity of $\widetilde{g}$,
 we can choose three sequences $\{\delta_{k}\}_{k\in \mathbb{N}}$,
 $\{\eta_{k}\}_{k\in \mathbb{N}}$ and $\{\lambda_{k}\}_{k\in \mathbb{N}}\subset (0\,,\,\infty)$ such that
 \begin{equation*}\label{eq3-10}
 0<\eta_{k+1}<\delta_{k}<u_{k}<\eta_{k},\quad \lim\limits_{k\to \infty}\eta_{k}=0,
 \quad \lim\limits_{k\to \infty}\lambda_{k}=0.
 \end{equation*}
 Moreover, for any $\lambda\in[-\lambda_{k}\,,\,\lambda_{k}]$ and every $k\in \mathbb{N}$, we have
 \begin{equation}\label{eq3-10*}
 g(u)\leq \widetilde{g}(\lambda\,,\,u)\leq 0\quad
 \forall u\in [\delta_{k}\,,\,\eta_{k}].
 \end{equation}

Define the function $g_{k}:\,[0\,,\,\infty)\to \mathbb{R}$ as follow
\begin{equation*}\label{eq3-13}
g_{k}(u)=g(\tau_{\eta_{k}}(u)),
\end{equation*}
where $\tau_{\eta}(t)=\min\{\,\eta\,,\,t\,\}$, $t>0$ and $g$ is given by \eqref{eq3-9}.
Then for every $k\in \mathbb{N}$ and  all $u\in[\delta_{k}\,,\,\eta_{k}] $,
from \eqref{eq3-10*}, we have
$$g_{k}(u)=g(\tau_{\eta_{k}}(u))=g(u)\leq 0,$$
which implies that the function $g_{k}$ verifies the hypotheses $(g_2)$ of Theorem \ref{th2-1}.
Moreover, the function $g$ is continuous and bounded on the set $[0\,,\,\eta_{1}]$,
so $(g_1)$ holds too.
Therefore, for every $k\in \mathbb{N}$ and $\lambda\in [-\lambda_{k}\,,\,\lambda_{k}]$,
 applying Theorem \ref{th2-1}, we get
 the following results:
\begin{itemize}
\item[(1)]there exists $u_{k,\lambda}\in E_{\eta_{k}}$such that
$\mathcal{E}_{k,\lambda}(u_{k,\lambda})=\min\limits_{u\in E_{\eta_{k}}}\mathcal{E}_{k,\lambda}$;
\item[(2)] $u_{k,\lambda}$ is a weak solution of Eq. \eqref{eq2-1} with $g=g_{k}$,
and $u_{k,\lambda}\in [0\,,\,\delta_{k}]$, where
\end{itemize}
 $$\mathcal{E}_{k,\lambda}(u)
=\frac{1}{2}\int\limits_{\mathbb{R}^{N}\times\mathbb{R}^{N}}\frac{|u(x)-u(y)|^{2}}{|x-y|^{N+2s}}dxdy+
\frac{\lambda_{0}}{2}\int\limits_{\Omega}|u|^{2}dx
-\int\limits_{\Omega}\int\limits_{0}^{u}g_{k}(t)dtdx.$$
Due to the definition of the functions $g_{k}$ and $\lambda_{0}>0$,
$u_{k,\lambda}$ is a weak solution not only for Eq.\eqref{eq2-1},
but also for our initial problem \eqref{eq1-1},
once that we guarantee that $u_{k,\lambda}\not\equiv0$.

{\bf Step 2.}
For $\lambda=0$, since $g(u)=\lambda_{0}\,u+f(u)$, the function
$g_{k}(u)$ also satisfies the hypotheses $(G_{1})$, $(G_2)$ and $(G_3)$ of Theorem \ref{th2-2}.
Consequently, similar as in the proof of the Theorem \ref{th2-2}, the solution $u_{k,0}$
 verifies  the results (1) and (2) in the Step 1,
and
\begin{equation}\label{eq3-18}
\mathcal{E}_{k,0}(u_{k,0})=\min\limits_{E_{\eta_{k}}}\mathcal{E}_{k,0}\leq
\mathcal{E}_{k,0}(z_{u_{k,0}})<0\quad \text{for every}\,\, k\in \mathbb{N},
\end{equation}
with
\begin{equation}\label{eq3-18*}
\lim\limits_{k\to \infty}\mathcal{E}_{k,0}(u_{k,0})=0,
\end{equation}
where $z_{u_{k,0}}\in E_{\eta_{k}}$ comes form \eqref{eq2-23} corresponding to $u_{k,0}$.

 Thanks to \eqref{eq3-18} and \eqref{eq3-18*},
 we may assume that there is a strickly increasing sequence
 $\{\theta_{k}\}_{k\in \mathbb{N}}$ such that $\lim\limits_{k\to +\infty}\theta_{k}=0$
and
\begin{equation*}\label{eq3-19}
\theta_{k}<\mathcal{E}_{k,0}(u_{k,0})\leq \mathcal{E}_{k,0}(z_{u_{k,0}})<\theta_{k+1}<0.
\end{equation*}

 Set
 $$
\alpha_{k}=\frac{(p+1)(\theta_{k+1}-\mathcal{E}_{k,0}(z_{u_{k,0}}))}{\|u_{k,0}\|^{p}_{L^{p}(\Omega)}}
$$
and
$$
\beta_{k}=\frac{(p+1)(\mathcal{E}_{k,0}(u_{k,0})-\theta_{k})}{\|u_{k,0}\|^{p}_{L^{p}(\Omega)}}.
$$
Thus $\alpha_{k}>0,\,\beta_{k}>0$ for all $k\in \mathbb{N}$.
Let $$\widetilde{\lambda}_{k}=\min\{\lambda_{1},\cdot\cdot\cdot,\lambda_{k},\alpha_{1},
\cdot\cdot\cdot,\alpha_{k},\beta_{1},\cdot\cdot\cdot,\beta_{k}\}>0.$$
Then for every $i\in\{1,2,\cdot\cdot\cdot,k\}$ and
$\lambda\in [-\tilde{\lambda}_{k}\,,\,\tilde{\lambda}_{k}]$ we have
\begin{equation}\label{eq3-21}
\aligned
\mathcal{E}_{i,\lambda}(u_{i,\lambda})
&\leq \mathcal{E}_{i,\lambda}(z_{u_{i,0}})\\
&=\frac{1}{2}\int\limits_{\mathbb{R}^{N}\times\mathbb{R}^{N}}\frac{| z_{u_{i,0}}(x)-z_{u_{i,0}}(y)|^{2}}{|x-y|^{N+2s}}dxdy+
\frac{\lambda_{0}}{2}\int\limits_{\Omega}|z_{u_{i,0}}|^{2}dx-\int\limits_{\Omega}\int\limits_{0}^{z_{u_{i,0}}}g_{i}(t)dtdx\\
&=\frac{1}{2}\int\limits_{\mathbb{R}^{N}\times\mathbb{R}^{N}}\frac{| z_{u_{i,0}}(x)-z_{u_{i,0}}(y)|^{2}}{|x-y|^{N+2s}}dxdy+
\frac{\lambda_{0}}{2}\int\limits_{\Omega}|z_{u_{i,0}}|^{2}dx-\int\limits_{\Omega}\int\limits_{0}^{z_{u_{i,0}}}g(t)dtdx\\
&=\frac{1}{2}\int\limits_{\mathbb{R}^{N}\times\mathbb{R}^{N}}\frac{| z_{u_{i,0}}(x)-z_{u_{i,0}}(y)|^{2}}{|x-y|^{N+2s}}dxdy
-\int\limits_{\Omega}F(z_{u_{i,0}})dx-\frac{\lambda}{p+1}\int\limits_{\Omega}|z_{u_{i,0}}|^{p+1}dx\\
&=\mathcal{E}_{0,i}(z_{u_{i,0}})-\frac{\lambda}{p+1}\int\limits_{\Omega}z_{u_{i,0}}^{p+1}dx\\
&<\theta_{i+1}.
\endaligned
\end{equation}

On the other hand,
taking into  account that $u_{i,\lambda}\in E_{\eta_{i}}$ and $u_{i,0}$
is the minimum point of functional $\mathcal{E}_{i,0}$ over the set $E_{\eta_{i}}$,
we get
\begin{equation}\label{eq3-22}
\aligned
\mathcal{E}_{i,\lambda}(u_{i,\lambda})
&=\mathcal{E}_{i,0}(u_{i,\lambda})-\frac{\lambda}{p+1}\int\limits_{\Omega}|u_{i,\lambda}|^{p+1}dx\\
&\geq \mathcal{E}_{i,0}(u_{i,0})-\frac{\lambda}{p+1}\int\limits_{\Omega}|u_{i,\lambda}|^{p+1}dx\\
&>\theta_{i}.
\endaligned
\end{equation}
Thus, from \eqref{eq3-21} and \eqref{eq3-22},
 for every $i\in\{1,2,\cdot\cdot\cdot,k\} $ and $\lambda\in [\,-\widetilde{\lambda}_{k}\,,\,\widetilde{\lambda}_{k}\,]$,
  we have
\begin{equation}\label{eq3-23}
\theta_{i}<\mathcal{E}_{i,\lambda}(u_{i,\lambda})<\theta_{i+1}<0.
\end{equation}
Moreover, we obtain
\begin{equation}\label{eq3-24}
\theta_{1}<\mathcal{E}_{1,\lambda}(u_{1,\lambda})<\theta_{2}<\mathcal{E}_{2,\lambda}(u_{2,\lambda})
<\cdot\cdot\cdot<\theta_{k}<\mathcal{E}_{k,\lambda}(u_{k,\lambda})<\theta_{k+1}<0,
\end{equation}
and, from
 $u_{i,\lambda}\in E_{\eta_{i}}\subset E_{\eta_{i-1}}\subset \cdot\cdot\cdot\subset E_{\eta_{1}} $,
$$\mathcal{E}_{i,\lambda}(u_{i,\lambda})=\mathcal{E}_{1,\lambda}(u_{i,\lambda})\quad \text{for every}\,\,i\in \{1,2,\cdot\cdot\cdot,k\}.$$
Therefore, from \eqref{eq3-24},
we obtain that for every $\lambda\in[\,-\widetilde{\lambda}_{k}\,,\,\widetilde{\lambda}_{k}\,],$
\begin{equation}\label{eq3-25}
\mathcal{E}_{1,\lambda}(u_{1,\lambda})
<\mathcal{E}_{1,\lambda}(u_{2,\lambda})\cdot\cdot\cdot<
\mathcal{E}_{1,\lambda}(u_{k-1,\lambda})<
\mathcal{E}_{1,\lambda}(u_{k,\lambda})<0
=\mathcal{E}_{1,\lambda}(0),
\end{equation}
which implies that $u_{1,\lambda}$, $\cdot\cdot\cdot$,
$u_{k,\lambda}$ are distinct and nontrivial positive solutions to problem \eqref{eq1-1} whenever $\lambda\in[-\widetilde{\lambda}_{k}\,,\,\widetilde{\lambda}_{k}]$.

{\bf Step 3.} It remains to prove  \eqref{eq1-5}.
We observe that for every $i\in\{1,2,\cdot\cdot\cdot,k\}$ and  $\lambda\in[\,-\widetilde{\lambda}_{k}\,,\,\widetilde{\lambda}_{k}\,]$,
\begin{equation}\label{eq3-26}
\mathcal{E}_{1,\lambda}(u_{i,\lambda})=
\mathcal{E}_{i,\lambda}(u_{i,\lambda})<\theta_{k+1}<0.
\end{equation}
Consequently, for every $k\in \{1,2,\cdot\cdot\cdot,\}$ and
$\lambda\in[-\widetilde{\lambda}_{k}\,,\,\widetilde{\lambda}_{k}]$, we obtain
\begin{equation*}\label{eq3-27}
\aligned
\frac{1}{2}\|u_{i,\lambda}\|_{\mX}^{2}
&<\frac{\lambda}{p+1}\int\limits_{\Omega}u_{i,\lambda}^{p+1}dx+\int\limits_{\Omega}F(u_{i,\lambda})dx\\
&\leq \Big{[}\frac{\tilde{\lambda}_{k}}{p+1}\eta_{1}^{p}mes(\Omega)+mes(\Omega)\max\limits_{u\in[0,1]}|f(u)|\Big{]}\delta_{k}\\
&\leq \frac{1}{2\,i^{2}},
\endaligned
\end{equation*}
where \begin{equation*}\label{eq3-11}
\delta_{i}\leq \min\Big{\{}\frac{1}{i}\,,\,
\frac{1}{2\,i^{2}}\Big(\frac{\tilde{\lambda}_{k}}{p+1}\eta_{1}^{p}mes(\Omega)
+mes(\Omega)\max\limits_{u\in [0\,,\,1]}|f(u)|\Big)^{-1}\Big{\}}.
\end{equation*}
This concludes the proof of Theorem \ref{th1-2}.

\vspace{10mm}

\section{Proofs of Theorems \ref{th1-3} and \ref{th1-4}}
In order to prove Theorem \ref{th1-3} and \ref{th1-4},
we follow more or less the technique of the previous Section 4.
For the completeness, we give all the details.

We consider again the problem
\begin{equation}\label{eq2-1-gg}
\left\{\begin{array}{lll}
&(-\Delta)^{s}u+\mu u=g_{k}(u),\,\,u\geq 0\quad &\text{in}\,\,\Omega,\\
&u=0\quad &\text{in}\,\,\mathbb{R}^{N}\setminus\Omega,\\
\end{array}\right.
\end{equation}
where the continuous functions $g $ fulfills $(G'_{1})$, $(G'_{2})$ and $(G'_{3})$.
Thus the results of Theorem \ref{th2-2} hold true for every $k\in \mathbb{N}$.
\vspace{3mm}

{\bf Proof of Theorem \ref{th1-3}.}
(a) Case $p=1$. For all $u\in [0\,,\,+\infty)$, let $\hat{\lambda}_{\infty}\in (\lambda_{\infty}\,,\,-l_{\infty})$ and choose
\begin{equation}\label{eq4-17}
\mu=\hat{\lambda}_{\infty}-\lambda\quad \text{and}
\quad g(u)=\hat{\lambda}_{\infty}\,u+f(u).
\end{equation}
It is clear that $\mu=\hat{\lambda}_{\infty}-\lambda\geq \hat{\lambda}_{\infty}-\lambda_{\infty}>0$.
Since $f(0)=0$ and the continuity of $f$ ,
we have $g(0)=0$ and  $\sup\limits_{u\in [0\,,\,t_{0}]}|g(u)|\in L^{\infty}$
for every $t_{0}>0$, i.e., $(G'_1)$ holds too.

Moreover,
 since $$\frac{G(u)}{u^2}=\frac{\hat{\lambda}_{\infty}}{2}+\frac{F(u)}{u^2},\quad u>0,$$
 where $G(u)=\int_{0}^{u}g(t)dt$, we have
 $$\liminf\limits_{u\to \infty}\frac{G(u)}{u^2}
 =\frac{\hat{\lambda}_{\infty}}{2}
 +\liminf\limits_{u\to \infty}\frac{F(u)}{u^2}$$
 and
 $$
 \quad\limsup\limits_{u\to \infty}\frac{G(u)}{u^2}
 =\frac{\hat{\lambda}_{\infty}}{2}
 +\limsup\limits_{u\to \infty}\frac{F(u)}{u^2}.$$
 Then $(f'_2)$ implies $(G'_2)$.

  Finally, from $\liminf\limits_{u\to \infty}\frac{f(u)}{u}=l_{\infty}<-\bar{\lambda}_{\infty}$
  and the continuity of $f$, there is a sequence
 $\{u_{k}\}_{k\in\mathbb{N}}\subset (0\,,\,+\infty)$ converging to $+\infty$ such that
 $$\frac{f(u_{k})}{u_{k}}<-\bar{\lambda}_{\infty}\quad \text{for all }\,\,k\in \mathbb{N}.$$
By the continuity of $f$, we may fix two sequences $\{\delta_{k}\}_{k\in\mathbb{N}}$
 and $\{\eta_{k}\}_{k\in\mathbb{N}}\subset (0\,,\,+\infty)$ satisfying
  $0<\delta_{k}<u_{k}<\eta_{k}<\delta_{k+1}$,
$\lim\limits_{k\to \infty}\delta_{k}=\infty$, and $\bar{\lambda}_{\infty}u+f(u)\leq 0$ for all
$u\in [\delta_{k}\,,\,\eta_{k}]$, $k\in \mathbb{N}$.
Therefore, $(G'_3)$ is also fulfilled.
Now, it remains to apply Theorem \ref{th2-3} by the
choice of \eqref{eq4-17} to get the assertion in Theorem \ref{th1-3} (a).

(b) Case $0<p<1$.  For all $u\in (0\,,\,+\infty)$ and let
\begin{equation}\label{eq4-18}
\mu=\lambda_{\infty}\quad\text{and}\quad g(u)=\lambda \,u^{p}+\lambda_{\infty}u+f(u),
\end{equation}
where $\lambda_{\infty}\in (0\,,\,-l_{\infty})$. It is easy to see that
$(G'_1)$ is satisfied.

On the other hand, since $0<p<1$ and
$$\frac{G(u)}{u^2}=\frac{\lambda }{p+1}u^{p-1}+\frac{\lambda_{\infty}}{2}+\frac{F(u)}{u^2},\quad u>0,$$
hypothesis $(f'_1)$ implies $(G'_2)$ holds too.

Moreover, for every $u\in (0\,,\,+\infty)$ we have
$$g(u)=\lambda \,u^{p}+\lambda_{\infty}u+f(u)\leq |\lambda|\,u^{p}+\lambda_{\infty}u+f(u).$$
Thanks to $0<p<1$, $\lambda_{\infty}<-l_{\infty}$ and $(f'_2)$, we have
$$\lim\limits_{u\to \infty}\frac{g(u)}{u}\leq
|\lambda| \lim\limits_{u\to \infty}u^{p-1}+\lambda_{\infty}+\lim\limits_{u\to \infty}\frac{f(u)}{u}=\lambda_{\infty}+l_{\infty}<0.$$
Therefore, there exists a sequence $\{u_{k}\}_{k\in \mathbb{N}}\subset (0\,,\,+\infty)$ converging to $+\infty$
such that $g(u_{k})<0$ for all $k\in \mathbb{N}$.
By the continuity of $g$, there exist two sequences $\{\delta_{k}\}_{k\in \mathbb{N}}$
 and $\{\eta_{k}\}_{k\in \mathbb{N}}\subset (0\,,\,+\infty)$ satisfying
 $0<\delta_{k}<u_{k}<\eta_{k}<\delta_{k+1}$, $\lim\limits_{k\to \infty}\delta_{k}=\infty$,
and $g(u)\leq 0$ for all $u\in [\delta_{k}\,,\,\eta_{k}]$, $k\in \mathbb{N}$.
Therefore, the hypotheses $(G'_3)$ holds for all $k\in \mathbb{N}$.
Now we can apply Theorem \ref{th2-3} to know that problem \eqref{eq1-1} is
equivalent to Eq. \eqref{eq2-1} by the choose of
\eqref{eq4-18}. We get the assertion.

\vspace{10mm}

{\bf Proof of Theorem \ref{th1-4}.}
The proof is divided into three steps.

{\bf Step 1.} \quad  Let $\lambda_{\infty}\in (0\,,\,-l_{\infty})$ and choose
$$\mu=\lambda_{\infty},\quad g(u)=\lambda \,u^{p}+\lambda_{\infty}\,u+f(u).$$

Define the function $\widetilde{g}:\,(-\infty\,,\,+\infty)\times[0\,,\,\infty)\to \mathbb{R}$ as follows:

\begin{equation}\label{eq4-22}
\widetilde{g}(\lambda,u)=|\lambda|\,u^{p}+\lambda_{\infty}\,u+f(u).
\end{equation}
Then we have $g(u)\leq \widetilde{g}(\lambda,u)$ for all $u\in[0\,,\,+\infty)$ and $\lambda\in \mathbb{R}$.
On account of $(f'_2)$ and $l_{\infty}<-\lambda_{\infty}$,
 there exists a sequence $\{u_{k}\}_{k\in \mathbb{N}}\subset (0\,,\,+\infty)$
 converging to $+\infty$
 such that $\frac{f(u_{k})}{u_{k}}<-\lambda_{\infty}$,
 that is, $\lambda_{\infty}\,u_{k}+f(u_{k})<0$ for all $u_{k}>0$.
 Therefore $\widetilde{g}(0,u_{k})<0$ for all $u_{k}>0$.

 By the continuity of $\widetilde{g}$, there exist
  three sequences $\{\delta_{k}\}_{k\in \mathbb{N}},\{\eta_{k}\}_{k\in \mathbb{N}}\subset (0\,,\,+\infty)$
  and $\{\lambda_{k}\}_{k\in \mathbb{N}}\subset (0\,,\,+\infty)$ such that
 $0<\delta_{k}<u_{k}<\eta_{k}<\delta_{k+1}$,
 $\lim\limits_{k\to \infty}\delta_{k}=\infty$,
 $\lim\limits_{k\to \infty}\lambda_{k}=0$ and for every $k\in \mathbb{N}$,
 \begin{equation}\label{eq4-23}
 \widetilde{g}(\lambda,u)\leq 0\quad \text{for all}\,\,\lambda\in[-\lambda_{k}\,,\,\lambda_{k}]\,\,
 \text{and}\,\,
 u\in [\delta_{k}\,,\,\eta_{k}].
 \end{equation}
Then taking into account of \eqref{eq4-23} and $g(u)\leq \widetilde{g}(\lambda,u)$, we have
 for every $k\in \mathbb{N}$ and for all $\lambda\in[-\lambda_{k}\,,\,\lambda_{k}]$
\begin{equation*}\label{eq4-24}
 g(u)\leq 0\quad \text{for all}\,\, u\in [\,\delta_{k}\,,\,\eta_{k}\,].
\end{equation*}

Define the function $g_{k}:\,[0\,,\,\infty)\to \mathbb{R}$ by
\begin{equation*}\label{eq4-26}
g_{k}(u)=g(\tau_{\eta_{k}}(u)),
\end{equation*}
where $\tau_{\eta}(t)=\min\{\eta\,,\,t\}$, $t>0$.
Then from $\delta_{k}<u_{k}<\eta_{k}$ for all $k\in \mathbb{N}$, we get
$$g_{k}(u)=g(\tau_{\eta_{k}}(u))=g(u)\quad \text{for all}\,\,u\in(\delta_{k}\,,\,\eta_{k}),$$
which implies that $g_{k}$ satisfies the hypotheses $(g_2)$ of Theorem \ref{th2-1}.
Therefore for every $k\in \mathbb{N}$ and $\lambda\in [-\lambda_{k}\,,\,\lambda_{k}]$,
applying  Theorem \ref{th2-1}, we obtain
\begin{itemize}
\item[(1)] there exists $ u_{k,\lambda}\in E_{\eta_{k}}$ such that
$\mathcal{E}_{k,\lambda}(u_{k,\lambda})=\min\limits_{u\in E_{\eta_{k}}}\mathcal{E}_{k,\lambda}$;
\item[(2)] $u_{k,\lambda}$ is a weak solution of Eq. \eqref{eq2-1} with $g=g_{k}$ and $u_{k,\lambda}\in [0\,,\,\delta_{k}]$.
\end{itemize}

Due to the definition of $g_{k}$ and $\mu$,
$u_{k,\lambda}$ is a  weak solution not only for Eq. \eqref{eq2-1},
 but also for our initial problem \eqref{eq1-1},
once we guarantee that $u_{k,\lambda}>0$ for all $k\in \mathbb{N}$.

{\bf Step 2.}\quad
Let $\lambda=0$, we get $g(u)=\lambda_{\infty}u+f(u)$.  Then
as in the proof of Theorem \ref{th1-3} (a) we have that the function
$g_{k}(u)=g(\tau_{\eta_{k}}(u))$ verifies the hypotheses of Theorem \ref{th2-3} whenever $\lambda=0$.
Consequently, there exists a sequence $\{u_{k,0}\}_{k\in \mathbb{N}}$  verifying the results
 (1) and (2) in Step 1.
 And moreover similar as in the proof of the Theorem \ref{th2-3},
 there exists a subsequence of $\{u_{k,0}\}_{k\in \mathbb{N}}$,
denoted by $\{u_{k_{i},0}\}_{i\in \mathbb{N}}$,
such that
\begin{equation}\label{eq4-31}
\mathcal{E}_{k_{i},0}(u_{k_{i},0})=\min\limits_{E_{\eta_{k_{i}}}}\mathcal{E}_{k_{i},0}\leq
\mathcal{E}_{k_{i},0}(z_{u_{k_{i}}}),\quad \forall \,\,i\in \mathbb{N},
\end{equation}
and
\begin{equation}\label{eq4-32}
\lim\limits_{i\to \infty}\mathcal{E}_{k_{i},0}(u_{k_{i},0})=-\infty.
\end{equation}
Without any loss of generality, we may assume that $k_{i}\geq i$ for all $i\in \mathbb{N}$. Then
from \eqref{eq4-31} and \eqref{eq4-32}, we may assume that there is a negative decreasing sequence
$\{\theta_{i}\}_{i\in \mathbb{N}}$,  up to a subsequence $\{\theta_{i}\}_{i\in\mathbb{N}}$,
such that
$\lim_{i\to +\infty}\theta_{i}=-\infty$
and
\begin{equation}\label{eq4-33}
\theta_{i+1}<\mathcal{E}_{k_{i},0}(u_{k_{i},0})\leq \mathcal{E}_{k_{i},0}(z_{u_{k_{i}}})<\theta_{i}.
\end{equation}

 Set
 $$
\alpha_{i}=\frac{(p+1)(\theta_{i}-\mathcal{E}_{k_{i},0}(z_{u_{k_{i}}}))}
{\delta^{p+1}_{k_{i}}}$$
and
$$\beta_{i}=\frac{(p+1)(\mathcal{E}_{k_{i},0}(u_{k_{i},0})-\theta_{i+1})}
{\delta^{p+1}_{k_{i}}},
$$
where $\{\delta_{k_{i}}\}_{i\in \mathbb{N}}$ is a subsequence of $\{\delta_{k}\}_{k\in \mathbb{N}}$
with $0<u_{k_{i},0}<\delta_{k_{i}}$ for every $i\in \mathbb{N}$.
Due to \eqref{eq4-33}, we get
 $$\tilde{\lambda}_{k}=\min\{\lambda_{1},\cdot\cdot\cdot,\lambda_{k},\alpha_{1},
\cdot\cdot\cdot,\alpha_{k},\beta_{1},\cdot\cdot\cdot,\beta_{k}\}>0.$$
Then for every $i\in\{1,2,\cdot\cdot\cdot,k\}$ and
$\lambda\in [-\tilde{\lambda}_{k}\,,\,\tilde{\lambda}_{k}]$, we have
\begin{equation}\label{eq4-36}
\aligned
\mathcal{E}_{k_{i},\lambda}(u_{k_{i}})
&\leq \mathcal{E}_{k_{i},\lambda}(z_{u_{k_{i},\lambda}})\\
&=\frac{1}{2}\int\limits_{\mathbb{R}^{N}\times\mathbb{R}^{N}}\frac{| z_{u_{k_{i}}}(x)- z_{u_{k_{i}}}(y)|^{2}}{|x-y|^{N+2s}}dxdy+
\frac{\lambda}{p+1}\int\limits_{\Omega}|z_{u_{k}}|^{p+1}dx-\int\limits_{\Omega}\int\limits_{0}^{z_{u_{k}}}g_{k}(t)dtdx\\
&=\mathcal{E}_{k_{i},0}(z_{u_{k_{i}}})-\frac{\lambda}{p+1}\int\limits_{\Omega}u_{k,\lambda}^{p+1}dx\\
&<\theta_{i}.
\endaligned
\end{equation}

On the other hand, since $u_{k_{i},\lambda}\in E_{\eta_{k_{i}}}$ and $u_{k_{i},0}$
is the minimum point of $\mathcal{E}_{k_{i},0}$ over the set $E_{\eta_{k_{i}}}$,
it yields
\begin{equation}\label{eq4-37}
\aligned
\mathcal{E}_{k_{i},\lambda}(u_{k_{i},\lambda})
&=\mathcal{E}_{k_{i},0}(u_{k_{i},\lambda})-\frac{\lambda}{p+1}\int\limits_{\Omega}|u_{k_{i},\lambda}|^{p+1}dx\\
&\geq \mathcal{E}_{k_{i}}(u_{k_{i},0})-\frac{\lambda}{p+1}\int\limits_{\Omega}|u_{m_k,\lambda}|^{p+1}dx\\
&>\theta_{i+1}.
\endaligned
\end{equation}
Then, for every $i\in\{1,2,\cdot\cdot\cdot,k\} $ and $\lambda\in [-\lambda_{k}\,,\,\lambda_{k}]$, from
\eqref{eq4-36} and \eqref{eq4-37} we have
\begin{equation*}\label{eq4-38}
\theta_{i+1}<\mathcal{E}_{k_{i},\lambda}(u_{k_{i},\lambda})<\theta_{i}<0,
\end{equation*}
and therefore
\begin{equation}\label{eq4-39}
\theta_{i+1}<\mathcal{E}_{k_{i},\lambda}(u_{k_{i},\lambda})<\theta_{i}<
\mathcal{E}_{k_{i-1},\lambda}(u_{k_{i-1},\lambda})<\cdot\cdot
\cdot<\theta_{2}<\mathcal{E}_{k_1,\lambda}(u_{k_1,\lambda})<\theta_{1}<0.
\end{equation}

Moreover, by the fact that $\eta_{k_{1}}<\eta_{k_{2}}<\cdot\cdot\cdot<\eta_{k_{i}}<\cdot\cdot\cdot$ and $u_{k_{i},\lambda}\in [0\,,\,\delta_{k_{i}}]\subset [0\,,\,\eta_{k_i}]$,
we get $u_{k_{i},\lambda}\in E_{\eta_{k_{i}}}\subset E_{\eta_{k_{k}}}$
for every $i\in \{1,2,\cdot\cdot\cdot,k\}$. So
\begin{equation*}\label{eq4-39**}
g_{k_{i}}(u_{k_{i},\lambda})=g(\tau_{\eta_{k_{i}}}(u_{k_{i},\lambda}))=
g(\tau_{\eta_{k_{i}}}(u_{k_{k},\lambda}))=g_{k_{i}}(u_{k_{k},\lambda}),
\end{equation*}
and
\begin{equation}\label{eq4-39*}
\mathcal{E}_{k_{i},\lambda}(u_{k_{i},\lambda})=\mathcal{E}_{k_{k},\lambda}(u_{k_{i},\lambda})\quad
\text{for every}\,\,i\in \{1,2,\cdot\cdot\cdot,k\}.
\end{equation}
Therefore, by \eqref{eq4-39} and \eqref{eq4-39*},
we obtain that for every $\lambda\in[-\lambda_{k}\,,\,\lambda_{k}],$
\begin{equation*}\label{eq4-40}
\mathcal{E}_{k_{k},\lambda}(u_{k_{k},\lambda})<
\mathcal{E}_{k_{k},\lambda}(u_{k_{k-1},\lambda})<
\cdot\cdot\cdot<
\mathcal{E}_{k_{k},\lambda}(u_{k_{2},\lambda})<
\mathcal{E}_{k_{k},\lambda}(u_{k_{1},\lambda})
<0=\mathcal{E}_{k_{k},\lambda}(0).
\end{equation*}
These inequalities show that the elements $u_{k_1,\lambda}$, $u_{k_2,\lambda}$, $\cdot\cdot\cdot$,
$u_{k_k,\lambda}$ are distinct and nontrivial solutions of problem \eqref{eq1-1} whenever $\lambda\in[-\lambda_{k}\,,\,\lambda_{k}]$.

{\bf Step 3.}\quad It remains to prove conclusion \eqref{eq1-6}.
We claim that $$\|u_{k_{i},\lambda}\|_{L^{\infty}(\Omega)}>\delta_{k_{i-1}}\quad
\text{for each} \,\,i\in\{2,3,\cdot\cdot\cdot,k\}.$$
Arguing by contradiction, we assume that there exists an $i_{0}\in\{2,3,\cdot\cdot\cdot,k\}$
such that $\|u_{k_{i_{0}},\lambda}\|_{L^{\infty}(\Omega)}\leq \delta_{k_{i_{0}-1}}$.
It follows from $\delta_{k_{i_0-1}}<\eta_{k_{i_0-1}}$ that $u_{k_{i_0},\lambda}\in E_{\eta_{k_{i_0-1}}}$.
Then
\begin{equation*}\label{eq4-41}
\mathcal{E}_{k_{i_0-1},\lambda}(u_{k_{i_0-1},\lambda})
=\min\limits_{E_{\eta_{k_{i_0-1}}}}\mathcal{E}_{k_{i_0-1},\lambda}
\leq \mathcal{E}_{k_{i_0-1},\lambda}(u_{k_{i_0},\lambda})
\leq \mathcal{E}_{k_{i_0},\lambda}(u_{k_{i_0},\lambda}),
\end{equation*}
which contradicts \eqref{eq4-39}. Therefore, for every $i\geq  2$,
without loss of generality, we assume that $k_{i}\geq i$, then
$$\|u_{k_{i},\lambda}\|_{L^{\infty}(\Omega)}\geq \delta_{k_{i-1}}\geq k_{i-1}\geq i-1.$$

Now, for $i=1$, from \eqref{eq4-39}, we have $\mathcal{E}_{k_1,\lambda}(u_{k_1,\lambda})<0$
which implies that \\
 $\|u_{k_{1},\lambda}\|_{L^{\infty}(\Omega)}>0$.
Thus the relation \eqref{eq1-6} holds true.
This completes the proof of Theorem \ref{th1-4}.

\section{Some more general problems}
In this section, we point out that the methods and results could be applied to some more general
problems. Firstly, let us consider the following problem involving concave-convex
and oscillating nonlinearities, that is,
\begin{equation}\label{eq5-1}
\left\{\begin{array}{lll}
&(-\Delta)^{s}u=\lambda u^{p}+\mu u^{q}+f(u),\,\,u>0 \quad &\mbox{in}\,\,\Omega,\\
&u=0\quad &\mbox{on}\,\,\mathbb{R}^{N}\setminus\Omega,\\
\end{array}\right.
\end{equation}
where $\Omega\subset \mathbb{R}^{N}$ is a bounded domain, $0<p<1<q$, and $\lambda,\,\mu\in \mathbb{R}$.
We have the following result.
\begin{theorem}\label{th3-1}
Assume $f\in C([0,\infty)\,,\,\mathbb{R})$, $f(0)=0$, $0<p<1<q$ and
\begin{itemize}
\item[(a)] If $(f_{0})$ holds, then for every $k\in \mathbb{N}$, and $\mu\in \mathbb{R}$,
there exists $\lambda_{k,\mu}>0$ such that Eq. \eqref{eq5-1} has at least $k$ distinctly positive solutions in $\mX$
whenever $\lambda\in [-\lambda_{k,\mu},\,,\lambda_{k,\mu}]$.
\item[(b)]If $(f_{\infty})$ holds, then for every $k\in \mathbb{N}$, and $\lambda\in \mathbb{R}$,
there exists $\mu_{k,\lambda}>0$ such that Eq. \eqref{eq5-1} has at least $k$ distinct positive solutions in $\mX$
whenever $\mu\in [-\mu_{k,\lambda},\,,\mu_{k,\lambda}]$.
\end{itemize}
\end{theorem}

Next, we will study the perturbed problem

\begin{equation}\label{eq5-2}
\left\{\begin{array}{lll}
&(-\Delta)^{s}u=f(u)+\varepsilon g(u),\,\,u>0 \quad &\mbox{in}\,\,\Omega,\\
&u=0\quad &\mbox{on}\,\,\mathbb{R}^{N}\setminus\Omega,\\
\end{array}\right.
\end{equation}
where $\varepsilon\geq 0$ and $\Omega$ is a bounded domain in $\mathbb{R}^{N}$ with smooth boundary.

The assumptions of \eqref{eq5-2} are replaced by
\begin{itemize}
\item[$(H_{fg})$] $f,\,g\in C([0\,,\,\infty)\,,\,\mathbb{R})$, such that $f(0)=g(0)=0$ and
 $\sup_{t\in [0\,,\,M]}(|f(t)|+|g(t)|)<\infty$
for some $M>0$.
\end{itemize}
\begin{itemize}
\item[$(f^{1}_{0})$]$-\infty<\liminf\limits_{t\to 0^{+}}\frac{F(t)}{t^2}\leq
                \limsup\limits_{t\to 0^{+}}\frac{F(t)}{t^2}=+\infty$ uniformly in $x\in \Omega$;
\item[$(f^{2}_{0})$] There exists a sequence $\{t_{k}\}_{k\in \mathbb{N}}\subset (0\,,\,+\infty)$
converging to 0 such that $t_{k}>0$,
and there exists $\alpha_{i}>0$ satisfying $f(t_{k})<-\alpha_{k}$ for all $x\in\Omega$
for every $k\in \mathbb{N}$;
\end{itemize}
and
\begin{itemize}
\item[$(f^{1}_{\infty})$]$-\infty<\liminf\limits_{t\to \infty}\frac{F(t)}{t^2}\leq
              \limsup\limits_{t\to\infty}\frac{F(t)}{t^2}=+\infty$ uniformly in $x\in \Omega$;
\item[$(f^{2}_{\infty})$] There exists a sequence $\{t_{k}\}_{k\in \mathbb{N}}\subset (0\,,\,+\infty)$ converging to $+\infty$ such that $t_{k}>0$,
and there exists $\beta_{i}>0$ satisfying $f(t_{k})<-\beta_{k}$ for all $x\in\Omega$ for every $k\in \mathbb{N}$.
\end{itemize}
Then we have similar conclusions corresponding to Theorems \ref{th1-1}-\ref{th1-4}.
We omit the proof here.

\begin{theorem}\label{th5-1}
Assume $(H_{fg})$, $(f^{1}_{0})$ and $(f^{2}_{0})$ hold.
\begin{itemize}
 \item[(i)]For $\varepsilon=0$, there exist infinitely many positive solutions
$\{u_{k}\}_{k}$ of Eq. \eqref{eq5-2} such that
$$\lim\limits_{k\to\infty}\|u_{k}\|_{L^{\infty}(\Omega)}=\lim\limits_{k\to \infty}\|u_{k}\|_{\mX}=0.$$
\item[(ii)] For every $k\in \mathbb{N}$, there exists $\varepsilon_{k}>0$ such that problem \eqref{eq5-2}
has at least $k$ solutions $u_{1,\varepsilon},\,u_{2,\varepsilon},\,\cdot\cdot\cdot,\,u_{k,\varepsilon}$
whenever $\varepsilon\in[-\varepsilon_{k}\,,\,\varepsilon_{k}]$. Moreover, for every $i\in\{1,2,\cdot\cdot\cdot,k\}$,
we have
$$\|u_{i,\varepsilon}\|_{\mX}<\frac{1}{i}\quad \text{and}\quad
\|u_{i,\varepsilon}\|_{L^{\infty}(\Omega)}<\frac{1}{i}\quad \forall \, \varepsilon\in [-\varepsilon_{k}\,,\,\varepsilon_{k}].$$
\end{itemize}
\end{theorem}

\begin{theorem}\label{th5-2}
Assume $(H_{fg})$, $(f^{1}_{\infty})$ and $(f^{2}_{\infty})$ hold.
\begin{itemize}
 \item[(i)]For $\varepsilon=0$, there exist infinitely many positive solutions
$\{u_{k}\}_{k}$ of Eq. \eqref{eq5-2} such that
$$\lim\limits_{k\to\infty}\|u_{k}\|_{L^{\infty}(\Omega)}=+\infty.$$
\item[(ii)] For every $k\in \mathbb{N}$, there exists $\varepsilon_{k}>0$ such that problem \eqref{eq5-2}
has at least $k$ solutions $u_{1,\varepsilon},\,u_{2,\varepsilon},\,\cdot\cdot\cdot,\,u_{k,\varepsilon}$
whenever $\varepsilon\in[-\varepsilon_{k}\,,\,\varepsilon_{k}]$. Moreover, for every $i\in\{1,2,\cdot\cdot\cdot,k\}$,
we have
$$\|u_{i,\varepsilon}\|_{L^{\infty}(\Omega)}>i-1\quad \forall \, \varepsilon\in [-\varepsilon_{k}\,,\,\varepsilon_{k}].$$
\end{itemize}
\end{theorem}

{\bf Acknowledgments} The authors thank Prof. Jianfu Yang for his guidance and useful discussions on the
topic of this paper.


\begin{thebibliography}{00}
\bibitem{a2009}
Applebaum, D.: L\'{e}vy processes and stochastic calculus, Second edition,
Cambridge Studies in Advanced Mathematics, 116. Cambridge
University Press, Cambridge, (2009)

\bibitem{b1996}
Bertoin, J.: L\'{e}vy processes. Cambridge Tracts in Mathematics, 121.
Cambridge University Press, Cambridge, (1996)

\bibitem{bes2012}
Barrios, B., Colorado, E., de Pablo, A., S\'{a}nchez, U.: On some critical problems for the fractional laplacian operator.
J. Diff. Equ. {\bf 252}, 6133--6162 (2012)


\bibitem{bes2013}
Barrios, B., Colorado, E., de Pablo, A., S\'{a}nchez, U.: A concave-convex elliptic problem involving the fractional Laplacian.
Proc. Royal Soc. Edi. {\bf 143A}, 39--71 (2013)

\bibitem{berf2014}
Barriosa, B., Coloradoc, E., Servadeid, R.,Soria, F. :
A critical fractional equation with concave每convex power nonlinearities.
Ann. I. H. Poincar\'e - AN (2014),
http://dx.doi.org/10.1016/j.anihpc.2014.04.0030294-1449/.2014

\bibitem{b2014}
Barrios, B., Peral, I., Soria, F.,Valdinoci, E.:
A widder's type theorem for the heat equation with nonlocal diffusion.
Arch. Rational Mech. Anal. {\bf 213}, 629--650 (2014)

\bibitem{ljy2010}
Caffarelli, L., Salsa, S., Sire, L.: Variational problems for free boundaries for the fractional Laplacian.
J. Eur. Math. Soc. {\bf 12}, 1151--1179 (2010)

\bibitem{lsl2008}
Caffarelli, L., Salsa, S., Silvestre, L.: Regularity estimates for the solution and the free boundary of
the obstacle problem for the fractional Laplacian. Invent. math. {\bf 171}, 425--461 (2008)

\bibitem{ll2007}
Caffarelli, L., Silvestre,L.: An extension problem related to the fractional Laplacian.
Comm. Part. Diff. Equ. {\bf 32},1245--1260 (2007)

\bibitem{cv2011}
Caffarelli, L., Valdinoci, E. : Uniform estimates and limiting arguments for nonlocal minimal surfaces.
Calc. Var. {\bf 41}, 203--240 (2011)

\bibitem{xt2010}
Cabr\'{e}, X., Tan, J. :
Positive solutions for nonlinear problems involving the square root of the Laplacian.
Adv. Math. {\bf 224}, 2052--2093 (2010)

\bibitem{cdds2011}
Capella, A., D\'avila, J., Dupaigne, L., Sire, Y.:
Regularity of radial extremal solutions for some non-local semilinear equations.
Comm. Partial Differential Equations {\bf 36}, 1353--1384 (2011)

\bibitem{cg2011}
Chang, S.Y.A., Gonz\'{a}lez, M.: Fractional Laplacian in conformal geometry. Adv. Math.
{\bf 226}, 1410--1432 (2011)

\bibitem{cw2014}
Chang, X., Wang, Z.-Q.: Nodal and multiple solutions of nonlinear problems involving the fractional laplacian.
J. Diff. Equ. {\bf 256}, 2965--2992 (2014)

\bibitem{cw2013}
Chang, X., Wang,Z.-Q.:
Ground state of scalar field equations involving a fractional Laplacian with general nonlinearity.
Nonlinearity {\bf 26}, 479--494 (2013)

\bibitem{ckl2013}
Choi, W., Kim, S., Lee,K.:
Asymptotic behavior of solutions for nonlinear elliptic problems with the fractional Laplacian.
J. Func. Anal. {\bf 266 (11)}, 6531--65981 (2014)

\bibitem{sp2013}
Di Nezza, E.,  Palatucci, G.,  Valdinoci, E.:  Hitchhiker's guide to the fractional Sobolev
spaces. Bull. Sci. Math.{\bf  136}, 521每-573(2012)

\bibitem{dpv2013}
 Dipierro, S.,  Palatucci, G.,   Valdinoci, E.:
 Existence and symmetry results for a Schr\"odinger type problem involving the fractional Laplacian.
Le Matematiche {\bf 68(1)}, 201--216 (2013)

\bibitem{fw2012}
Fall, M. M., Weth, T.: Nonexistence results for a class of fractional elliptic boundary value
problems. J. Funct. Anal. {\bf 263}, 2205每-2227 (2012)

\bibitem{ff2014}
Fall, M. M., Felli, V.:
Unique continuation property and local asymptotics o solutions to fractional elliptic equation.
Comm. Partial Differential Equations {\bf 39}, 354--397 (2014)


\bibitem{gz2013}
Guo, Z.:
multiple solutions for perturbed quasilinear elliptic problems with oscillatory terms.
Nonlinear Anal. {\bf 77}, 149--157 (2013)

\bibitem{ak2008}
 Krist\'{a}ly, A.:
 Detection of arbitrarily many solutions for perturbed elliptic problems involving oscillatory terms.
 J. Diff. Equ. {\bf 245},  3849--3868 (2008)

\bibitem{kmt2007}
 Krist\'{a}ly, A., Moro\c{s}anu, Gh., Tersian, S.:
 Quasilinear elliptic problems in $\mathbb{R}^{N}$ involving oscillatory nonlinearities.
 J. Diff. Equ.{\bf 235}(2), 366--375 (2007)

\bibitem{km2010}
 Krist\'{a}ly, A.,  Moro\c{s}anu, Gh.:
New competition phenomena in Dirichlet problems. J. Math. Pures Appl. {\bf 94},  555--570 (2010)


\bibitem{v2011}
Maz'ya, V.: Sobolev Spaces with Applications to Elliptic Partial Differential Equations,
second, revised and augmented edition.  Grundlehren Math. Wiss. (Fundamental Principles of Mathematical Sciences),
vol.342, Springer, Heidelberg, (2011)

\bibitem{oo2006}
 Obersnel,F.,Omari, P.:
 Positive solutions of elliptic problems with locally oscillating nonlinearities. J. Math. Anal. Appl. {\bf 323},
913--929 (2006)

\bibitem{oz1996}
 Omari, P., Zanolin,F.:
 Infinitely many solutions of a quasilinear elliptic problem with an oscillatory potential. Comm. Partial Differential
Equations {\bf 21},  721--733 (1996)

\bibitem{r2002}
 Raymond, J. S.: On the multiplicity of the solutions of the equation $-\Delta u=\lambda f(u)$.
 J. Differential Equations {\bf 180},  65--88 (2002)

\bibitem{s2006}
Silvestre, L.:  Regularity of the obstacle problem for a fractional power of the Laplace operator. Comm. Pure Appl.
Math. {\bf 60}, 67--112(2006)

\bibitem{sr2014-1}
Serra,J., Ros-Oton, X.:  The Dirichlet problem for the fractional Laplacian: regularity up
to the boundary. J. Math. Pures Appl. {\bf 101}, no. 3, 275302 (2014)

\bibitem{sr2014-2}
Serra, J., Ros-Oton, X.: The Pohozaev identity for the fractional Laplacian. arXiv:1207.5986

\bibitem{s2013-1}
Servadei, R.:  The Yamabe equation in a non-local setting. Adv. Nonlinear Anal. Vol.2
235--270(2013)

\bibitem{s2014-1}
Servadei, R.:  A critical fractional Laplace equation in the resonant case. to appear in Topol.
Methods Nonlinear Anal.

\bibitem{s2013-2}
Servadei, R., Valdinoci, E.:  Lewy-Stampacchia type estimates for variational inequalities
driven by (non)local operators. Rev. Mat. Iberoam.{\bf 29}, 1091--1126 (2013)


\bibitem{s2012}
 Servadei, R., Valdinoci, E.:  Mountain Pass solutions for non-local elliptic operators. J.
Math. Anal. Appl.{\bf 389},887每-898 (2012)

\bibitem{sv2014}
Servadei, R., Valdinoci, E.:
  The Brezis-Nirenberg result for the fractional Laplacian.Trans. Amer. Math. Soc. in press (2014)

\bibitem{sv2013}
 Servadei, R., Valdinoci, E. :
  A Brezis-Nirenberg result for non-local critical equations in low dimension.
Commun. Pure Appl. Anal. {\bf 12(6)}, 2445--2464 (2013)




\bibitem{t2011}
Tan, J.:  The Brezis-Nirenberg type problem involving the square root of the Laplacian.
Calc. Var. {\bf 42}, 21--41(2011)




\bibitem{ws2014}
Wei, Y., Su, X.:  Multiplicity of solutions for non-local elliptic equations driven by the fractional Laplacian.
Calc. Var. DOI 10. 1007/s00526-013-0706-5




\bibitem{z2014}
Zhang, J.:
Three solutions for a fractional elliptic problems with critical and supercritical growth.
arXiv:1404.7361v1, (2014)


\bibitem{zl2014}
Zhang, J., Liu, X.:
Positive solutions to some asymptotically linear fractional Schr\"odinger equations.
arXiv:1411.2189v1, (2014)








\end{thebibliography}
\end{document}